\documentclass[11pt]{amsart}

\usepackage{mathrsfs,pstricks,pst-plot}

\usepackage[dvips]{epsfig}

\parskip=\smallskipamount

\newtheorem{theorem}{Theorem}[section]
\newtheorem{lemma}[theorem]{Lemma}
\newtheorem{corollary}[theorem]{Corollary}
\newtheorem{proposition}[theorem]{Proposition}

\theoremstyle{definition}
\newtheorem{definition}[theorem]{Definition}
\newtheorem{example}[theorem]{Example}

\newtheorem{remark}[theorem]{Remark}

\numberwithin{equation}{section}

\newcommand{\C}{\mathbb{C}}

\newcommand{\N}{\mathbb{N}}

\newcommand{\Z}{\mathbb{Z}}
\renewcommand{\P}{\mathbb{P}}
\newcommand{\R}{\mathbb{R}}

\def\Ascr{\mathscr{A}}
\def\Oscr{\mathscr{O}}

\newcommand{\cA}{\mathcal{A}}

\newcommand{\cH}{\mathcal{H}}

\newcommand{\cP}{\mathcal{P}}

\newcommand{\cT}{\mathcal{T}}
\newcommand{\cU}{\mathcal{U}}

\newcommand{\jgoth}{{\ensuremath{\mathfrak{j}}}}
\newcommand{\igoth}{{\ensuremath{\mathfrak{i}}}}

\newcommand{\IA}{\mathscr I_A(M)}

\newcommand\wt{\widetilde}
\newcommand\hra{\hookrightarrow}

\newcommand\di{\partial}

\newcommand\wh{\widehat}

\newcommand{\ol}{\overline}

\usepackage{amssymb}
\input xy
\xyoption{all}

\begin{document}
\title
{Null curves and directed immersions of open Riemann surfaces}
\author{Antonio Alarc\'on and Franc Forstneri\v c}
\address{A.\ Alarc\'{o}n, Departamento de Geometr\'{\i}a y Topolog\'{\i}a, 
Universidad de Gra\-na\-da, E-18071 Granada, Spain}
\email{alarcon\,@ugr.es}

\address{F.\ Forstneri\v c, Faculty of Mathematics and Physics, University of Ljubljana, and Institute of Mathematics, Physics and Mechanics, Jadranska 19, 1000 Ljubljana, Slovenia}
\email{franc.forstneric\,@fmf.uni-lj.si}

%
%

\date{31 March 2013}


\maketitle


\section{Introduction} 
\label{sec:Intro}
This paper was motivated by open problems in a classical field of geometry --- {\em null curves} in $\C^3$ and in $SL_2(\C)$. The former are holomorphic immersions $F=(F_1,F_2,F_3)\colon M\to\C^3$ of an open Riemann surface $M$ into $\C^3$ which are directed by the quadric subvariety
\begin{equation}
\label{eq:null}
	A=\bigl\{z = (z_1,z_2,z_3)\in \C^3\colon z_1^2+z_2^2+z_3^2=0 \bigr\},
\end{equation}
in the sense that the derivative $F'=(F'_1,F'_2,F'_3)$ with respect to any local holomorphic coordinate on $M$ has range in $A\setminus \{0\}$. The real and the imaginary part of a null curve are minimal surfaces in $\R^3$; conversely, every simply connected, conformally immersed minimal surface in $\R^3$ is the real part of a null curve in $\C^3$. This connection has strongly influenced the theory of minimal surfaces, supplying this field with powerful tools coming from complex analysis and Riemann surfaces theory. (See \cite{Osserman} for a classical survey of this subject and \cite{MP1,MP2} for recent ones.) Similarly, null curves in $SL_2(\C)$ are holomorphic immersions $M\to SL_2(\C)$ directed by the variety
\begin{equation}
\label{eq:nullSL}
	\left \{ z= \left(\begin{matrix} z_{11} & z_{12} 
	\cr z_{21} & z_{22} \end{matrix} \right) \colon \det z= z_{11}z_{22}-z_{12}z_{21}=0\right \} 
	\subset \C^4.
\end{equation}
The projection of a null curve in $SL_2(\C)$ to the hyperbolic 3-space $\cH^3=SL_2(\C)/SU(2)$ is a {\em Bryant surface} (mean curvature 1 surface); conversely, a simply connected Bryant surface  in $\cH^3$ lifts to a null curve in $SL_2(\C)$ \cite{Br}. See for instance \cite{UY,CHR,Ro} for the background on this topic. 

In spite of the rich literature on null curves, many basic problems remain open. In this paper we invent new methods to solve several of them, not only for null curves, but also for immersions with derivative in an arbitrary conical subvariety $A$ of $\C^n$ which is smooth away from the origin; such directed immersions will be called {\em $A$-immersions} (Def.\ \ref{def:directed}). For convenience we also assume that $A$ is irreducible and is not contained in any hyperplane. We point out that null curves in $\C^3$, and in $\C^n$ for any $n\ge 3$, are a particular case of all our results. Unlike in many papers where results hold only after a deformation of the complex structure on the Riemann surface $M$, usually due to cutting away pieces of the surface (see among others \cite{AL1,AL3}, and \cite{AFM,Al,FMM} for the corresponding problems on minimal surfaces), we always work with a fixed complex structure. 

Our results can be divided in two classes; lacking a better term, we call them local and global ones. The local results pertain to compact bordered Riemann surfaces. We show that the set of all $A$-immersions $M\to\C^n$ contains an open everywhere dense subset connsisting of so called {\em nondegenerate $A$-immersions} (Def.\ \ref{def:nondegenerate}), and this subset is an infinite dimensional complex Banach manifold; see Theorem \ref{th:local}. We also prove that every $A$-immersion can be approximated by $A$-embeddings 
(i.e., injective $A$-immersions); see Theorem \ref{th:desing} and Corollary \ref{co:sl2c-desing}. 

The global results pertain to all open Riemann surfaces, but we assume in addition that $A\setminus\{0\}$ is an {\em Oka manifold} in the sense of \cite{FF:Oka}. (See Sect.\ \ref{sec:Okamanifolds} below.) This condition is completely natural, and its 1-dimensional version, Property $\mathrm{CAP}_1$ (see \cite[Def.\ 5.4.3]{F:book}), is even necessary. This holds in particular for the varieties (\ref{eq:null}), (\ref{eq:nullSL}) controlling null curves, and for any other quadric conical hypersurface in $\C^n$ which is smooth away from the origin. We prove that $A$-immersions can be approximated by $A$-embeddings (Theorem \ref{th:desing2}), and they satisfy the Oka principle (Theorem \ref{th:Oka}), including the Runge and the Mergelyan approximation property (Theorems \ref{th:Mergelyan} and \ref{th:Mergelyan2}). For such $A$, we use these tools to construct proper $A$-embeddings of an arbitrary open Riemann surface into $\C^n$; see Theorem \ref{th:proper}. 


Directed immersions have been studied in many classical geometries (symplectic, contact, totally real, lagrangian, etc.); surveys can be found in the monographs by Gromov \cite{Gromov:PDR} and Eliashberg and Mishachev \cite{EM} (see in particular Chap.\ 19 in the latter). Apart from specific examples such as null curves, ours seems to be the first systematic investigation of this subject in the holomorphic case. Interesting new problems open up, and we point out some of them at the end of the following section.


\section{Main results} 
\label{sec:results}
In this section we present our main results, indicate the methods used in the proof, and mention some interesting problems that our work opens. The organization of the paper is explained along the way.

\begin{definition}\label{def:directed}
Let $A$ be a (topologically) closed conical complex subvariety of $\C^n$. (A subvariety $A$ is conical if  $tA=A$ for every $t\in \C^*=\C\setminus\{0\}$.) A holomorphic immersion $F=(F_1,\ldots, F_n) \colon M\to \C^n$ of an open Riemann surface $M$ to $\C^n$ is said to be {\em directed by $A$}, or an {\em $A$-immersion}, if its complex derivative $F'$ with respect to any local holomorphic coordinate on $M$ assumes values in $A\setminus \{0\}$. Likewise, $F$ is said to be an {\em $A$-embedding} if it is an injective $A$-immersion.
\end{definition}

By Chow's theorem, every closed conical complex subvariety of $\C^n$ is algebraic and is the common zero set of finitely many homogeneous holomorphic polynomials; see e.g.\ \cite[p.\ 73]{Chirka}.

Definition \ref{def:directed} also applies if $M$ is a bordered Riemann surface with smooth boundary $bM\subset M$ and $F\colon M\to\C^n$ is a map of class $\mathscr C^1(M)$ (i.e., continuously differentiable on $M$) which is holomorphic in the interior 
$\mathring M=M\setminus bM$.

Pick a nowhere vanishing holomorphic 1-form $\theta$ on $M$.  (Such a 1-form exists by the Oka-Grauert principle, see   \cite{Gra1,Gra2,Gra3} or Theorem 5.3.1 in \cite[p.\ 190]{F:book}. If $M$ is a compact bordered Riemann surface, we can choose $\theta$ to be smooth on $M$ and holomorphic in the interior.) Then we can write $dF=f\theta$, where $f=(f_1,\ldots,f_n)\colon M\to \C^n$ is a holomorphic map. Clearly $F$ is an $A$-immersion if and only if $f=dF/\theta$ maps $M$ to $A\setminus\{0\}$. The specific choice of the 1-form $\theta$ does not matter since $A$ is conical. 

We shall always assume that $A\setminus \{0\}$ is smooth (non-singular). Without loss of generality we also assume that $M$, and the submanifold $A\setminus \{0\}\subset \C^n$, are connected (so the variety $A$ is irreducible), and that $A$ is not contained in any hyperplane of $\C^n$. These conditions imply that $n\ge 3$ (since the only irreducible complex cones in $\C^2$ are complex lines).

We denote by $\IA$ the set of all $A$-immersions of an open Riemann surface $M$ to $\C^n$. If $M$ is a compact bordered Riemann surface with smooth boundary $\emptyset\ne bM \subset M$, we denote by $\mathscr A^r(M,\C^n)$ the set of all maps $M\to \C^n$ of class $\mathscr C^r$ $(r\in \Z_+)$  that are holomorphic on $\mathring M$, and by  
${\mathscr I}_A(M) \subset \Ascr^1(M,\C^n)$ the set of all $A$-immersions $M\to \C^n$ of class $\mathscr C^1$ that are holomorphic on 
$\mathring M$. The space $\IA$ is naturally endowed with the topology associated to the $\mathscr C^1$ maximum norm on $M$ (see Sec.\ \ref{sec:prelim}).

The following notion will play an important role in our analysis.

\begin{definition}
\label{def:nondegenerate}
An $A$-immersion $F\colon M\to \C^n$ is said to be {\em nondegenerate} if the linear span of the tangent spaces $T_{f(x)}A$, $x\in M$, equals $\C^n$; otherwise $F$ is said to be {\em degenerate}. (Here $dF=f\theta$ as above.)
\end{definition}

It is immediate that nondegenerate $A$-immersions of a compact bordered Riemann surface $M$ form an open subset of the space $\IA$. Our first result concerns the local structure of $\IA$.

\begin{theorem} 
\label{th:local}
Let $M$ be a compact bordered Riemann surface, and let $A$ be an irreducible closed conical subvariety of $\C^n$ $(n\ge 3)$ which is not contained in any hyperplane and such that $A\setminus \{0\}$ is smooth.  Then the following hold: 
\begin{itemize}
\item[\rm(a)]
	Every $A$-immersion $F\in \IA$ can be approximated in the ${\mathscr C}^1(M)$ topology by non\-de\-generate $A$-immersions.
\item[\rm (b)] 
The set of all nondegenerate $A$-immersions $M\to\C^n$ is a complex Banach manifold.
\item[\rm (c)] If $M$ is a smoothly bounded compact domain in a Riemann surface $R$, then every $F\in \IA$ can be approximated in the ${\mathscr C}^1(M)$ topology by $A$-immersions defined on small open neighborhoods of $M$ in $R$.
\end{itemize}
\end{theorem}

Observe that $\IA$ is always nonempty as it contains immersions of the form $M \ni x \mapsto z g(x)\in\C^n$, where $z\in A\setminus \{0\}$ and $g$ is a holomorphic function without critical points on $M$ (see Gunning and Narasimhan \cite{GN}). Immersions of this form are obviously degenerate.
 
Theorem \ref{th:local} is proved in Sec.\ \ref{sec:local}. In relation to part (b), we can add that the set of all nondegenerate $A$-immersions is a {\em split complex Banach submanifold} of the Banach space $\mathscr A^1(M,\C^n)$; see Remark \ref{rem:split}.
 
The following desingularization result is new even for null curves in $\C^3$.

\begin{theorem}
\label{th:desing}
Let $M$ be a compact bordered Riemann surface. If $A\subset \C^n$ $(n\ge 3)$ is a subvariety as in Theorem \ref{th:local}, then every $A$-immersion $M\to\C^n$ can be approximated in the ${\mathscr C}^1(M)$ topology by $A$-embeddings. 
\end{theorem}

Theorem \ref{th:desing} is proved in Sec.\ \ref{sec:desing} by using the transversality theorem. To construct a submersive family of $A$-immersions to which Sard's lemma applies, we use methods developed in the proof of Theorem \ref{th:local}, but with rather precise estimates. 

Theorem \ref{th:desing} and the results of \cite{AL1} immediately give complete properly embedded null curves in any convex domain of $\C^3$; see Corollary \ref{co:CY} in Sec.\ \ref{sec:desing} for a brief discussion of this result.

We now pass on to the global results, assuming in addition that $A\setminus\{0\}$ is an {\em Oka manifold}. (See Def.\ \ref{def:Oka} in Sec.\ \ref{sec:Okamanifolds} below, 
where we also recall some sufficient geometric conditions for this property.) 
The quadric (\ref{eq:null}) controlling null curves is Oka. We mention in particular that $A\setminus \{0\}$ is Oka if and only if the projection $A_\infty \subset \C\P^{n-1}$ of $A$ to the hyperplane at infinity is Oka; see Proposition \ref{prop:Oka}. 

The following analogue of Theorem  \ref{th:desing} is proved in Sec.\ \ref{sec:desing}.

\begin{theorem}
\label{th:desing2}
Let $M$ be an open Riemann surface, and let $A\subset \C^n$ be a closed conical subvariety as in Theorem \ref{th:local}. If $A\setminus \{0\}$ is an Oka manifold, then every $A$-immersion $M\to\C^n$ can be approximated uniformly on compacts by $A$-embeddings. 

In particular, every immersed null curve in $\C^n$, $n\ge 3$, can be approximated by embedded null curves.
\end{theorem}

Next we describe the {\em Oka principle for directed immersions}. 

Recall that a compact set $K$ in a complex manifold $M$ is said to be {\em $\Oscr(M)$-convex} if $K$ equals its holomorphically convex hull 
\[
	\wh K=\{x\in M\colon |f(x)| \le \sup_K |f|\ \ \forall f\in \Oscr(M)\}. 
\]
If $M$ is a Stein manifold (for instance, an open Riemann surface), then $K=\wh K$ implies the Runge theorem, also called the Oka-Weil theorem in this setting: Every holomorphic function in a neighborhood of $K$ can be approximated, uniformly on $K$, by functions holomorphic on $M$. (See e.g.\ \cite{Hormander-SCV}.) For this reason we shall also call such a set $K$ a {\em Runge set} in $M$.

\begin{theorem}
\label{th:Oka}
{\rm (The Oka principle for $A$-immersions.)} 
Let $M$ be an open Riemann surface, and let $A\subset \C^n$ $(n\ge 3)$ be a closed conical subvariety as in Theorem \ref{th:local}. Assume in addition  that $A\setminus \{0\}$ is an Oka manifold (see Def.\ \ref{def:Oka} below). Fix a nowhere vanishing holomorphic $1$-form $\theta$ on $M$. 

Every continuous map $f\colon M\to A\setminus\{0\}$ is homotopic to a holomorphic map $\tilde f\colon M\to A\setminus\{0\}$ such that $\tilde f\theta=d\wt F$ is the differential of an $A$-immersion $\wt F\colon M\to \C^n$. Furthermore, if $f\theta=dF$ is the differential of an $A$-immersion $F$ on an open  neighborhood of a compact Runge set $K\subset M$, then the A-immersion $\wt F$ can be chosen to approximate $F$ uniformly on $K$. 
\end{theorem}

Theorem \ref{th:Oka} is proved in Sec.\ \ref{sec:Oka}. In the proof we combine the Oka principle for maps to Oka manifolds (see Theorem \ref{th:OkaP}) with the technique of controlling the periods, developed in the proof of Theorem \ref{th:local}. The Oka property of $A\setminus \{0\}$ allows us to deform $f$ to a holomorphic map $\tilde f\colon M\to A\setminus\{0\}$. The main nontrivial point is to show that this map can be chosen such that the 1-form $\tilde f\theta$ has vanishing periods, i.e., it is exact. This is done inductively by passing through sublevel sets of a strongly subharmonic Morse exhaustion function. The control of the periods is provided by Lemma \ref{lem:convex}.

If $D$ is a closed Jordan domain in an open Riemann surface $M$, then $M$ is obtained by successively adding $0$- and $1$-handles to $D$. It follows that any continuous map $D\to X$ to a connected manifold $X$ extends to a continuous map $M\to X$. In the context of Theorem \ref{th:Oka}, with $X=A\setminus \{0\}$ and $D$ a suitable open neighborhood of the compact Runge set $K$ in $M$, we obtain the following corollary to Theorem \ref{th:Oka}.

\begin{corollary}
\label{cor:Runge}
{\em (Runge theorem for $A$-immersions.)}
Let $M$ and $A\subset \C^n$ be as in Theorem \ref{th:Oka}. Assume that $K$ is a compact Runge set in $M$ and $F\colon U\to \C^n$ is an $A$-immersion on an open neighborhood of $K$. Then $F$ can be approximated uniformly on $K$ by $A$-immersions $M\to \C^n$.
\end{corollary}

We also obtain the {\em Mergelyan approximation theorem} for $A$-immer\-sions; see Theorems \ref{th:Mergelyan} and \ref{th:Mergelyan2} in Sec.\ \ref{sec:Oka}. In the second version (Theorem \ref{th:Mergelyan2}) we assume that the directional variety $A$ has a smooth hyperplane section which is an Oka manifold, and we prove the Mergelyan theorem for $A$-immersions with a fixed component function that is holomorphic on all of $M$. This enables us to show that every open Riemann surface carries a proper $A$-embedding into $\C^n$ (see Theorem \ref{th:proper}), thereby solving a natural question that has been open even for null curves in $\C^3$.

We remark that, for null curves in $\C^3$, Runge and Mergelyan theorems were proved by Alarc\'on and L\'opez \cite{AL2}. Their analysis depends on the {\em Weierstrass representation} of a null curve, a tool that is not available in the general situation considered here.


We end this survey of results by briefly discussing null curves in 
\[
   SL_2(\C)	=\left \{ z= \left(\begin{matrix} z_{11} & z_{12} \cr z_{21} & z_{22} \end{matrix} \right) \colon \det z= z_{11}z_{22}-z_{12}z_{21}=1\right \}\subset\C^4.
\]
Although the quadric variety \eqref{eq:nullSL} meets all our requirements, our methods do not apply directly to null curves in $SL_2(\C)$. Indeed, applying our deformation procedures to a null curve $M\to SL_2(\C)\subset \C^4$, one gets holomorphic immersions $M\to\C^4$ directed by the variety \eqref{eq:nullSL}, but the resulting curves need not lie in $SL_2(\C)$. However, some of our results can be transported to null curves in $SL_2(\C)\setminus\{z_{11}=0\}$ by using the biholomorphism
$\cT\colon\C^3\setminus\{z_3=0\} \to SL_2(\C)\setminus\{z_{11}=0\}$, given by
\begin{equation}\label{eq:T}
\cT(z_1,z_2,z_3) = \frac{1}{z_3}\left(\begin{matrix} 1 & z_1+\imath z_2 \cr z_1-\imath z_2 & z_1^2+z_2^2+z_3^2
\end{matrix} \right), \quad \imath=\sqrt{-1},
\end{equation}
which maps null curves into null curves; see \cite{MUY1}. The following is an immediate corollary to Theorem \ref{th:desing}.

\begin{corollary}\label{co:sl2c-desing}
If $M$ is a compact bordered Riemann surface, then every immersed null curve $M\to SL_2(\C)\setminus\{z_{11}=0\}$ can be approximated in the ${\mathscr C}^1(M)$ topology by embedded null curves.   
\end{corollary}

The local Mergelyan approximation theorem holds in $SL_2(\C)\setminus\{z_{11}=0\}$; see Corollary \ref{co:sl2c-mergelyan}. Finally, Corollary \ref{co:CY} and the correspondence $\cT$ give complete bounded embedded null curves in $SL_2(\C)$. 

On the other hand, we do not know how to derive the Oka principle and the global approximation pro\-perties for null curves in $SL_2(\C)\setminus\{z_{11}=0\}$ by using the correspondence $\cT$. Furthermore, $\cT$ does not seem useful for determining whether every open Riemann surface is a properly embedded null curve in $SL_2(\C)$, in analogy to the situation in $\C^3$ (see Theorem \ref{th:proper}). This question remains open even allowing self-intersections.


\smallskip\noindent\it Methods used in the proof. \rm
We systematically work with the derivative map $f\colon M\to A\setminus \{0\}$, where $dF=f\theta$ and $\theta$ is a nowhere vanishing holomorphic 1-form on $M$. Every small holomorphic deformation of $f$ is obtained as a composition of flows of holomorphic vector fields tangential to $A$, where the respective time variables are holomorphic functions on $M$. The main point is to find deformations with suitable properties and with vanishing periods over a basis of the 1st homology group $H_1(M;\Z)$; the latter property ensures that the map integrates to an $A$-immersion $M\to\C^n$. For this purpose we develop an effective method of controlling the periods (see Lemma \ref{lem:deformation}). Although similar  technique have been used before (see e.g.\ \cite{GN,LMM,Majcen-2007,AL1}), our proof pertains to a completely general situation and only assumes that $A$ is not contained in any complex hypersurface. This implies that the convex hull of $A$ equals $\C^n$ (see Lemma \ref{lem:convexhull}). Our use of this condition, which corresponds to {\em ampleness} in Gromov's theory, is reminiscent of the proof of the {\em convex integration lemma}, the basic step in Gromov's method of convex integration of partial differential relations \cite{Gromov:convex,Gromov:PDR,EM}. 

In our global results, the method of controlling the periods is combined with the Oka principle for maps from Stein manifolds to Oka manifolds (see Sect.\ \ref{sec:Okamanifolds}). Finally, to desingularize $A$-immersions to $A$-embeddings, we combine all of the above techniques with Abraham's approach \cite{Abraham} to transversality theorems.

\smallskip\noindent\it Open problems. \rm
Our results open several directions of possible further research; let us mention a few of them. 

First of all, one could study {\em $A$-maps} $F\colon M\to\C^n$, that is, holomorphic maps whose derivative belongs to $A$, but may assume the value zero. The analysis near singularities of an immersion may be rather delicate. 

An interesting generalization would be to allow the variety $A$ to depend on the base point. Let $A$ be a closed complex subvariety of $T\C^n \cong \C^n\times\C^n$ with conical fibers $A_z\subset T_z\C^n$. (A special case with linear fibers are holomorphic distributions.) A holomorphic immersion $F\colon M\to\C^n$ is an $A$-immersion if $F'(x)\in A_{F(x)}$ for every $x\in M$. To what extent do our results generalize to this setting? Since sections of the first coordinate projection $\pi\colon A\to \C^n$ define vector fields on $\C^n$ whose integral curves are $A$-immersions, there exist plenty of $A$-immersions of the disc, but it may be difficult to deal with non-simply connected Riemann surfaces. An example are the complex Legendrian curves in $\C^3$ which are directed by the distribution $dz_1+z_3dz_2$. 

Finally, what could be said about directed holomorphic immersions of Stein manifolds of dimension $>1$? Due to involutivity obstructions one must in general allow an open set of directions (an {\em open differential relation} in Gromov's terminology \cite{Gromov:PDR}) to obtain nontrivial results. We hope to return to these interesting questions in a future work.

%
\section{Preliminaries}
\label{sec:prelim}
In this section we establish the notation and recall some basic facts on Riemann surfaces; see for example \cite{AS,Forster-book} or any other standard source. 

Let $M$ be a compact connected bordered Riemann surface with boundary $\emptyset \ne bM \subset M$. Denote by $g$ the genus of $M$ and by $m$ the number of boundary components of $M$. By gluing a disc onto $M$ along each of its boundary curves, we obtain a compact surface, $R$, containing $M$ as a domain with smooth boundary; the number $g$ is also the genus of $R$. 

The 1-st homology group $H_1(M;\Z)$ is a free abelian group on $l=2g+m-1$ generators. We can represent the basis of $H_1(M;\Z)$ by closed, smoothly embedded loops $\gamma_1,\ldots,\gamma_l\colon S^1 \to \mathring M$ that only meet at a chosen base point $p\in \mathring M$. Let $C_j=\gamma_j(S^1)\subset M$ denote the trace of $\gamma_j$. Their union $C=\bigcup_{j=1}^l C_j$ is a wedge of $l$ circles meeting at $p$. 

For any $r\in \Z_+$ we denote by $\mathscr A^r(M)=\mathscr A^r(M,\C)$ the space of all ${\mathscr C}^r$ functions $M\to \C$ that are holomorphic in $\mathring M$. Similary we define the space $\Ascr^r(M,X)$ of maps $A\to X$ to a complex manifold $X$. We write $\Ascr^0(M)=\Ascr(M)$ and $\Ascr^0(M,X)=\Ascr(M,X)$. The space $\Ascr^r(M,X)$ is naturally endowed with the topology associated to the ${\mathscr C}^r$ maximum norm, measured with respect to some smooth Riemannian metric on each of the two manifolds. Note that $\Ascr^r(M,\C^n)$ is a complex Banach space; furthermore, for any complex manifold $X$, $\Ascr^r(M,X)$ is a complex Banach manifold modeled on $\Ascr^r(M,\C^n)$ with $n=\dim X$ (see \cite[Theorem 1.1]{FF:manifolds}).

A function $f\colon M\to \C$ is said to be holomorphic on $M$ if it is holomorphic on some unspecified open neighborhood of $M$ in $R$; the space of all such functions is denoted $\Oscr(M)$. Holomorphic 1-forms on $M$ are defined likewise, as the restrictions to $M$ of holomorphic 1-forms in open neighborhoods on $M$ in $R$. It is classical that each function in $\Ascr^r(M)$ can be approximated in the ${\mathscr C}^r(M)$ topology by functions in $\Oscr(M)$. The same is true for maps to an arbitrary complex manifold or complex space (see \cite[Theorem 5.1]{DF}).

Let $A$ be a closed conical complex subvariety of $\C^n$. By Chow's theorem, such a variety $A$ is algebraic. Let $\overline A\subset \C\P^n=\C^n\cup \C\P^{n-1}$ be the projective closure of $A$. Denoting by $0$ the origin of $\C^n$, we have a holomorphic vector bundle projection $\pi\colon \C\P^n\setminus\{0\} \to \C\P^{n-1}$. Then 
\begin{equation}
\label{eq:projA}
	A_\infty := \pi(A\setminus \{0\}) = \pi(\overline A \setminus\{0\}) = \overline A\cap \C\P^{n-1}
\end{equation}
is a closed algebraic subvariety of $\C\P^{n-1}$, and $\overline A$ is the cone over $A_\infty$. Note that $A\setminus \{0\}$ is smooth and connected if and only if $A_\infty$ is such. Furthermore, $A$ is not contained in any hyperplane in $\C^n$ if and only if $A_\infty$ is not contained in any projective hyperplane in $\C\P^{n-1}$. If $A\subset \C^3$ is the quadric variety (\ref{eq:null}) controlling null curves, then $A_\infty$ is a smooth rational curve  in $\C\P^2$ (an embedded copy of the Riemann sphere $\C\P^1$).

The following observation will be used in Sec.\ \ref{sec:Oka} below; lacking a precise reference, we include a short proof. Simple examples show that the result fails in general for non-algebraic subvarieties.

\begin{lemma}
\label{lem:convexhull}
The convex hull of an algebraic subvariety $A\subset \C^n$ is the smallest affine complex subspace of $\C^n$ containing $A$.
\end{lemma}

\begin{proof}
If the convex hull of $A$ is not all of $\C^n$, the Hahn-Banach theorem tells us that $A$ lies in the half space $\{u\le c\}$ for some nonconstant real linear function $u$ on $\C^n$. Let $\overline A \subset \C\P^n$ denote the projective closure of $A$. Since $u$ is bounded from above on $A$, it extends across the subvariety $A_\infty=\overline A \cap \C\P^{n-1}$ to a bounded plurisubharmonic function $u^*\colon \overline A \to \R\cup\{-\infty\}$. As $\overline A $ is compact, the maximum principle implies that $u^*$ is constant, so $A$ lies in an affine complex hyperplane. Now proceed inductively. 
\end{proof}

Fix a nowhere vanishing holomorphic 1-form $\theta$ on $M$; such exists since every holomorphic vector bundle on $M$ is trivial by the Oka-Grauert principle \cite{Gra1,Gra2,Gra3}. Given a holomorphic map $F=(F_1,\ldots,F_n)\colon M\to\C^n$, we write $dF_j=f_j\theta$ and identify the differential $dF$ with the map $f=(f_1,\ldots,f_n)\colon M\to \C^n$. Then $F$ is an $A$-immersion if and only if $f$ maps $M$ to $A\setminus\{0\}$. Conversely, a holomorphic map $f=(f_1,\ldots,f_n)\colon M\to A\setminus \{0\}$ determines an $A$-immersion $F\colon M\to\C^n$ if and only if the holomorphic vector-valued 1-form $f\theta=(f_1\theta,\ldots,f_n\theta)$ is exact, and in this case $F$ is obtained as the integral
\begin{equation}
\label{eq:integral}
	  F(x)= F(p) + \int_{p}^x f\theta , \quad x\in M
\end{equation}
from an arbitrary initial point $p\in M$ to $x$. The choice of the integration curve from $p$ to $x$ does  not matter.

For a fixed  choice of a nowhere vanishing holomorphic 1-form $\theta$ on $M$ and of the basis $\{\gamma_j\}_{j=1}^l$ of the homology group $H_1(M;\Z)$, we denote by 
\[
	\cP=(\cP_1,\ldots, \cP_l)\colon \Ascr(M,\C^n)\to (\C^n)^l
\]
the {\em period map} whose $j$-th component, applied to $f\in \Ascr(M,\C^n)$, equals 
\begin{equation}
\label{eq:P}
	\cP_j(f) = \int_{\gamma_j} f\theta = \int_0^1 f(\gamma_j(t))\, \theta(\gamma_j(t),\dot{\gamma_j}(t))\, dt
	\in \C^n.
\end{equation}
By Stokes' theorem, the period map does not change under homotopic deformations of the loops $\gamma_j \colon [0,1]\to M$, and the 1-form $f\theta$ is exact if and only if its periods vanish: $\cP_j(f) = \int_{\gamma_j} f\theta =0$ for $j=1,\ldots,l$.

%
\section{Oka manifolds} 
\label{sec:Okamanifolds}

In this section we recall the Oka principle for maps of Stein manifolds to Oka manifolds, and we mention the geometric conditions and examples which are most relevant for us and which illuminate the scope of our global results on $A$-immersions. A comprehensive exposition of Oka theory can be found in \cite{F:book}; for a more leisurely introduction see the surveys \cite{FL1,FL2}. The paper \cite{FL3} contains up-to-date information of what is known about which compact complex surfaces are Oka.

The concept of an Oka manifold evolved from the classical Oka-Grauert principle and the seminal work of Gromov \cite{Gromov:Oka}. This class of manifolds was first formally introduced in \cite{FF:Oka}; see also \cite[Def.\ 5.4.1]{F:book}. 

%
%
\begin{definition}
\label{def:Oka}  
A complex manifold $X$ is said to be an {\em Oka manifold} if every holomorphic map from a neighborhood of a compact convex set $K\subset \C^N$ to $X$ can be approximated, uniformly on $K$, by entire maps $\C^N\to X$.
\end{definition}

The main result is that maps $M \to X$ from a Stein manifold $M$ to an Oka manifold $X$ satisfy all forms of the Oka principle (see \cite[Theorem 5.4.4]{F:book}). We shall use the following {\em Oka property with approximation}.

%
%
\begin{theorem}
\label{th:OkaP}
Assume that $X$ is an Oka manifold, and let $K$ be a compact Runge set in a Stein manifold $M$. Then every continuous map $f\colon M\to X$ which is holomorphic on a neighborhood of $K$ can be approximated, uniformly on $K$, by holomorphic maps $M\to X$ that are homotopic to $f$.
\end{theorem}

By a theorem of Grauert \cite{Gra1}, every complex homogeneous manifold is an Oka manifold. (See also \cite[Proposition 5.5.1]{F:book}.) The most useful geometric conditions which are known to imply that a manifold is Oka are {\em ellipticity} in the sense of Gromov \cite{Gromov:Oka}, and {\em subellipticity} in the sense of Forstneri\v c \cite{FF:subelliptic}. These conditions are formulated in terms of (families of) holomorphic sprays $s\colon E\to X$ on holomorphic vector bundles $E\to X$. A spray restricts to the identity map on the zero section of $E$; it is said to be {\em dominating} if its derivative $ds_{0_z} \colon E_z \to T_z X$ in the fiber direction at any point $0_z$ in the zero section is surjective. A complex manifold which admits a dominating spray is said to be elliptic. Similarly one defines dominability of a family of sprays and subelliptic manifolds.

\begin{example}
{\rm (Gromov \cite{Gromov:Oka})}
If the tangent bundle $TX$ of a complex manifold $X$ is pointwise spanned by finitely many $\C$-complete holomorphic vector fields $V_j$ $(j=1,\ldots,m)$, then the compositions of their flows $\phi^j_t$ gives a dominating spray on $X$, defined on the trivial bundle of rank $m$ over $X$:
\begin{equation}
\label{eq:flow-spray}
	s(z,t_1,\ldots,t_m)=\phi^1_{t_1}\circ \phi^2_{t_2}\circ \cdots \circ \phi^m_{t_m} (z),\quad
	z\in X,\ (t_1,\ldots,t_m)\in\C^m.
\end{equation}
Hence every such manifold is elliptic and therefore Oka.
\qed \end{example}

We shall apply Theorem \ref{th:OkaP} to the manifold $X=A\setminus\{0\}$, where $A \subset\C^n$ is a conical algebraic subvariety as in Theorem \ref{th:local}. Such $A$ is a complex cone over the projective manifold $\pi(A)= A_\infty \subset \C\P^{n-1}$; see (\ref{eq:projA}).

\begin{example}
\label{eq:quadrics}
If $P(z)$ is a homogeneous quadratic polynomial on $\C^n$ such that the conical hypersurface $A=\{P=0\}$ is smooth away from the origin, then the manifold $A\setminus \{0\}$  admits a dominating spray of the form (\ref{eq:flow-spray}), and hence is an Oka manifold.  Indeed, the holomorphic vector fields 
\[
	V_{j,k}=\frac{\di P}{\di z_j}\frac{\di}{\di z_k}-\frac{\di P}{\di z_k}\frac{\di}{\di z_j},
	\quad 1\le j\ne k\le n,
\]
are tangential to $A$, they span the tangent space $T_zA$ at each point $0\ne z\in A$, they are linear and hence $\C$-complete, and their flows preserve $A\setminus \{0\}$. 

In particular, the varieties (\ref{eq:null}), (\ref{eq:nullSL}) determining null curves in $\C^3$ and $SL_2(\C)$, respectively, are Oka manifolds after removing the origin. 
\qed\end{example}

Since the projection $\pi\colon A \setminus \{0\} \to A_\infty\subset\C\P^{n-1}$ is a holomorphic fiber bundle with Oka fiber $\C^*$, Theorem 5.5.4 in \cite{F:book} implies the following.

\begin{proposition}
\label{prop:Oka}
Let $A$ be a closed conical subvariety of $\C^n$ such that $A\setminus\{0\}$ is smooth. Then $A\setminus\{0\}$ is an Oka manifold if any only if $A_\infty=\overline A\cap \C\P^{n-1}$ \eqref{eq:projA} is an Oka manifold. 
\end{proposition}

\begin{example}
\label{ex:OkaRS}
A Riemann surface $X$ is Oka if and only if it is not Kobayashi hyperbolic 
(see \cite{Gra1} and \cite[Corollary 5.5.3]{F:book}). If such $X$ is compact, it is either the Riemann sphere $\C\P^1$ or a complex torus. Each of these surfaces embeds in $\C\P^k$ for any $k\ge 2$; embeddings of $\C\P^1$ are called {\em rational curves}, while embeddings of tori are {\em elliptic curves}. The complex cone $A\setminus \{0\} \subset \C^{k+1}$ over any such curve is an Oka manifold in view of Proposition \ref{prop:Oka}.  
\qed \end{example}

%
%
%
%

%
\section{Local structure of the space $\IA$}
\label{sec:local}
In this section we prove Theorem \ref{th:local}. We use the notation established in Sec.\ \ref{sec:prelim}. In particular, we fix a nowhere vanishing holomorphic 1-form $\theta$ on $M$ and let $\cP$ denote the associated period map given by (\ref{eq:P}). 

\smallskip
\noindent \it Proof of part (a). \rm 
Assume that $F\colon M\to \C^n$ is a degenerate $A$-immersion (see Def.\ \ref{def:nondegenerate}). Write $dF=f\theta$, with $f\colon M\to A\setminus\{0\}$ and $\cP(f)=0$. Let $\Sigma(f)$ denote the $\C$-linear subspace of $\C^n$ spanned by all the tangent spaces $T_z A$, $z\in f(M)$. Since $F$ is degenerate, $\Sigma(f)$ is a proper subspace of $\C^n$. To prove part (a), it suffices to find a map $\wt f \colon M\to A$ of class $\Ascr(M)$, arbitrarily close to $f$, such that $\cP(\wt f)=0$ and $\dim \Sigma(\wt f) > \dim \Sigma(f)$; the proof is then finished by a finite induction. 

Choose points $x_1,\ldots,x_k\in M$ such that the tangent spaces $T_{f(x_j)} A$ for $j=1,\ldots,k$ span $\Sigma(f)$. The set 
\[
	A'=\{0\} \cup\{ z\in A\setminus\{0\} \colon T_z A \subset \Sigma(f)\}
\]
is a proper complex subvariety of $A$, and we have $f(M)\subset A'\setminus\{0\}$. 

Choose a holomorphic tangential vector field $V$ on $A$ that vanishes at $0$ and is not everywhere tangential to $A'$ along $f(M)$. Let $t\mapsto \phi(t,z)$ denote the flow of $V$ for small complex values of time $t$, with $\phi(0,z)=z \in A$. Choose a nonconstant function $h\in \Oscr(M)$ that vanishes at the points $x_1,\ldots, x_k$. For any function $g$ in a small open neighborhood of the zero function in $\Ascr(M)$ we define the map $\Phi(g)\in \Ascr(M,A\setminus\{0\})$ by setting
\[
	\Phi(g)(x)= \phi(g(x)h(x), f(x)),\quad x\in M.
\]
Clearly $\Phi(g)$ depends holomorphically on $g$. Consider the holomorphic map
\[
	\Ascr(M) \ni g \longmapsto \cP(\Phi(g)) \in (\C^n)^l.
\]
Since $\cP(\Phi(0))= \cP(f)=0$ and the space $\Ascr(M)$ is infinite dimensional, there is a nonconstant function $g\in \Ascr(M)$ arbitrarily close to $0$ such that $\cP(\Phi(g))=0$. For such $g$, the map $\tilde f = \Phi(g) \colon M\to A\setminus \{0\}$ integrates to an $A$-immersion $\wt F\in \IA$ that is close to $F$. 

Since the function $gh$ vanishes at the points $x_1,\ldots, x_k$, we have $\Phi(g)(x_j)=\phi(0,f(x_j))=f(x_j)$ for $j=1,\ldots,k$, so $\Sigma(f) \subset \Sigma(\tilde f)$.  Furthermore, as $gh$ is nonconstant on $M$ and the vector field $V$ is not everywhere tangential to $A'$ along $f(M)$, there is a point $x_0\in M \setminus \{x_1,\ldots,x_k\}$ such that $z_0:=\Phi(g)(x_0)\in A\setminus A'$. By the definition of $A'$, the tangent space $T_{z_0} A \subset \Sigma(\tilde f)$ is not contained in $\Sigma(f)$, so $\Sigma(\tilde f)$ is strictly larger than $\Sigma(f)$. This concludes the proof of part (a) in Theorem \ref{th:local}.

\smallskip
\noindent \it Proof of part (b). \rm 
According to \cite[Theorem 1.1]{FF:manifolds}, $\Ascr(M,A\setminus\{0\})$ is a complex Banach manifold modeled on the complex Banach space $\Ascr(M,\C^k)$ with $k=\dim A$. The following lemma is the key to the proof of part (b).

\begin{lemma}
\label{lem:deformation}
Let $F\in \IA$ be a nondegenerate $A$-immersion, and write $dF=f\theta$ with $f\colon M\to A\setminus \{0\}$. Then there exist an open neighborhood $U$ of the origin in some $\C^N$ and a holomorphic map 
\[
	U\times M\ni (\zeta,x) \longmapsto \Phi_f(\zeta,x)\in A\setminus\{0\} 
\]
such that $\Phi_f(0,\cdotp)=f$ and the period map $\zeta\mapsto \cP(\Phi_f(\zeta,\cdotp)) \in (\C^n)^l$ (\ref{eq:P}) has maximal rank equal to $ln$ at $\zeta=0$. 

Furthermore, there is a neighborhood $V$ of $f$ in $\Ascr(M,A\setminus\{0\})$ such that the map $V\ni \tilde f \mapsto \Phi_{\tilde f}$ can be chosen to depend holomorphically on $\tilde f$.
\end{lemma}

Assuming Lemma \ref{lem:deformation} for a moment, we explain how it implies Theorem \ref{th:local}-(b). Denote by $\Ascr^*(M,A\setminus\{0\})$ the set of all $f\in \Ascr(M,A\setminus\{0\})$ satisfying $\cP(f)=0$, where $\cP$ is the period map (\ref{eq:P}). Let us call $f\in \Ascr(M,A\setminus\{0\})$ {\em nondegenerate} if it satisfies the condition in Definition \ref{def:nondegenerate}, and let $\Ascr^*_{reg}(M,A\setminus\{0\})$ denote the set of all nondegenerate maps with vanishing periods (an open subset of $\Ascr^*(M,A\setminus\{0\})$). The components of the period map $\cP$ are holomorphic functions on $\Ascr(M,A\setminus\{0\})$, and Lemma \ref{lem:deformation} says that their differentials are linearly independent at every point $f_0\in \Ascr^*_{reg}(M,A\setminus\{0\})$. The implicit function theorem shows that any such $f_0$ admits an open neighborhood $U\subset \Ascr(M,A\setminus\{0\})$ such $U\cap \Ascr^*(M,A\setminus\{0\})$ is a complex Banach submanifold of $U$, parametrized by the kernel of the differential $d\cP_{f_0}$ of the period map at $f_0$ (a complex codimension $ln$ subspace of the Banach space $\Ascr(M,\C^k)$ with $k=\dim A$). This shows that $\Ascr^*_{reg}(M,A\setminus\{0\})$ is a complex Banach manifold. The integration (\ref{eq:integral}), with an arbitrary choice of the initial value $F(p)\in \C^n$, provides an isomorphism between $\Ascr^*_{reg}(M,A\setminus\{0\}) \times \C^n$ and the subset of $\IA$ consisting of all nondegenerate $A$-immersions, so Theorem \ref{th:local}-(b) follows. 
  

\smallskip 
\noindent \it Proof of Lemma \ref{lem:deformation}. \rm
Let $C_1,\ldots, C_l$ be smooth embedded loops in $M$ forming a basis of $H_1(M;\Z)$, and set $C=\bigcup_{j=1}^l C_j \subset M$. As explained in Sec.\ \ref{sec:prelim}, we may assume that the $C_j$'s only meet at a common point $p\in M$ and are otherwise pairwise disjoint. Let $\gamma_j\colon [0,1]\to C_j$ be a parametrization of $C_j$. 

By Cartan's theorem A there exist holomorphic vector fields $V_1,\ldots, V_m$ on $\C^n$ that are tangential to $A$ along $A$, that vanish at $0\in A$, and that satisfy $\mathrm{span}(V_1(z),\ldots, V_m(z))=T_z A$ for every $z\in A\setminus \{0\}$. We may assume that $m\ge n$. Let $\phi^j_t$ denote the flow of $V_j$ for small complex values of time $t$. Since $V_j$ is tangential to $A$, we have $\phi^j_t(x)\in A$ when $x\in A$.

For every $i=1,\ldots,l$ let $h_{i,1},\ldots, h_{i,m} \colon C \to \C$ be smooth functions that are identically zero except on $C_i$; their values on $C_i$ will be specified later. In particular, these functions vanish near the common point $p$ of the $C_i$'s. 

Let $\zeta=\left(\zeta_1,\ldots, \zeta_l\right)\in \C^{lm}=(\C^m)^l$, where  $\zeta_j=(\zeta_{j,1},\ldots,\zeta_{j,m})\in \C^m$ are complex coordinates on the $j$-th copy of $\C^m$ in the above product. For a sufficiently small open neighborhood $U\subset \C^{lm}$ of the origin we define a smooth map $\Psi\colon U \times C \to A$ by setting
\begin{equation}
\label{eq:Psi}
	\Psi(\zeta,x)=\phi^1_{\zeta_{1,1}h_{1,1}(x)} \circ \cdots  \circ \phi^m_{\zeta_{l,m}h_{l,m}(x)} (f(x))
\end{equation}
for $\zeta\in U\subset \C^{lm}$ and $x\in C$. The composition on the right hand side of (\ref{eq:Psi}) contains all terms $\phi^j_{\zeta_{j,k}h_{j,k}(x)}$ (the flow of $V_j$ for time $\zeta_{j,k}h_{j,k}(x)$) for $j=1,\ldots,l$ and $k=1,\ldots, m$. The order of terms in $\Psi$ is unimportant. Note that 
\[
	\Psi(0,x)= f(x),\quad x\in C.
\]
Since the function $h_{i,k}$ vanishes on $C\setminus C_i$, it follows that 
\[
	\Psi(\zeta,x) = \phi^1_{\zeta_{i,1}h_{i,1}(x)} \circ \cdots  \circ \phi^m_{\zeta_{i,m}h_{i,m}(x)} (f(x)) 	
		\quad\text{for}\ x\in C_i, 
\]
and $\Psi(\zeta,x)$ is independent of $\zeta_{j,k}$ when $x\in C_i$ and $i\ne j$. Note that $\Psi$ is smooth, and it is holomorphic in the variable $\zeta\in U$ for every fixed $x\in C$. For any $x\in C_i$ its partial derivatives at $\zeta=0$ equal
\begin{eqnarray*}
	\frac{\di \Psi(\zeta,x)}{\di \zeta_{i,k}}\bigg|_{\zeta=0} &=& h_{i,k}(x)\, V_k(f(x)), \cr
	\frac{\di \Psi(\zeta,x)}{\di \zeta_{j,k}}\bigg|_{\zeta=0} &=& 0 \qquad \text{for}\ j\ne i. \cr 
\end{eqnarray*}

Recall that $\cP$ is the period map (\ref{eq:P}). Consider the map $P=\cP(\Psi)=(P_1,\ldots,P_l)\colon U\to (\C^n)^l$ defined by 
\begin{equation}
\label{eq:P2}
	P_i(\zeta) = \cP_i(\Psi(\zeta,\cdotp)) = \int_{C_i} \Psi(\zeta,\cdotp)\,\theta = 
	\int_0^1 \Psi(\zeta,\gamma_i(t))\, \theta(\gamma_i(t),\dot\gamma_i(t)) \,dt 
\end{equation}
for $\zeta \in U$ and $i=1,\ldots,l$. That is, $P_i(\zeta)\in\C^n$ is the period of the map $C_i\ni x\mapsto \Psi(\zeta,x)\in A$ around the loop $C_i$. Clearly $P$ is holomorphic. Its partial derivatives at $\zeta=0$ equal
\begin{eqnarray*}
	\frac{\di P_i(\zeta)}{\di \zeta_{i,k}}\bigg|_{\zeta=0} &=& 
	\int_0^1 \frac{\di \Psi(\zeta,\gamma_i(t))}{\di \zeta_{i,k}}\bigg|_{\zeta=0}\,
	 \theta(\gamma_i(t),\dot\gamma_i(t)) \,dt \cr
	&=& \int_0^1 h_{i,k}(\gamma_i(t))\, V_k(f(\gamma_i(t)) \, \theta(\gamma_i(t),\dot\gamma_i(t)) \,dt; \cr
	\frac{\di P_i(\zeta)}{\di \zeta_{j,k}}\bigg|_{\zeta=0} &=& 0 \quad \text{for}\ j\ne i.\cr
\end{eqnarray*}
Hence the matrix representing the differential $dP(0)$ has a block structure, with $l\times l$ blocks, each of them of size $n\times m$, such that the blocks off the main diagonal are zero, while the $i$-th diagonal block represents the partial differential $d_{\zeta_i}P_i(0)$. The $k$-th column of this block equals $\frac{\di P_i(\zeta)}{\di \zeta_{i,k}}\big|_{\zeta=0}$.

\begin{lemma}
\label{lem:maxrank}
The functions $h_{j,k}\colon C\to \C$ can be chosen such that the differential $dP(0)$ of the period map (\ref{eq:P2}) at the origin has maximal rank $ln$.
\end{lemma}

\begin{proof}
Due to the block structure of $dP(0)$ as explained above, it suffices to ensure that each of the diagonal blocks of size $n\times m$, representing the partial differentials $d_{\zeta_i}P_i(0)$, has maximal rank $n$. 

Write $dF=f\theta$ as before. Since the immersion $F$ is nondegenerate, the tangent spaces $T_{f(x)} A$ over all $x\in M$ span $\C^n$. By the identity principle for holomorphic functions on $M$, the same is true if we restrict the point $x$ to any nontrivial curve in $M$; in particular, to the loop $C_i$. Hence there exist points $x_{i,1},\ldots, x_{i,m} \in C_i\setminus \{p\}$ such that the vectors $V_k(f(x_{i,k}))$ for $k=1,\ldots,m$ span $\C^n$. Let $t_{i,1},\ldots, t_{i,m}\in (0,1)$ be such that $\gamma_i(t_{i,k})=x_{i,k}$ for $k=1,\ldots,m$. For every $k$ we choose a smooth function $\eta_{i,k}\colon [0,1]\to \R_+$, supported in a small neighborhood of $t_{i,k}$, such that $\int_0^1 \eta_{i,k} \, dt=1$. Let $h_{i,k} \colon C_i\to\R_+$ be defined by $h_{i,k}(\gamma_i(t))=\eta_{i,k}(t)$ for $t\in [0,1]$, and extend it to $C$ by setting $h_{i,k}=0$ on $C\setminus C_i$. We then have
\[
	\int_0^1 h_{i,k}(\gamma_i(t))\, V_k(f(\gamma_i(t))) \, \theta(\gamma_i(t),\dot\gamma_i(t)) \,dt  \approx
  V_k(f(x_{i,k}))\, \theta(\gamma_i(t_{i,k}),\dot\gamma_i(t_{i,k}))
\]
for $k=1,\ldots,m$. Since $\theta(\gamma_i(t_{i,k}),\dot\gamma_i(t_{i,k}))\ne 0$, the vectors on the right hand side span $\C^n$, and hence the same is true for the vectors on the left hand side, provided that all approximations are close enough. For such choices of the $h_{i,k}$'s the partial differential $d_{\zeta_i}P_i(0)$ has maximal rank $n$.
\end{proof}

By Mergelyan's theorem we approximate each of the functions $h_{i,k}$, uniformly on $C$, by holomorphic functions $g_{i,k}\colon M\to \C$. Replacing the functions $h_{i,k}$ by $g_{i,k}$ in the expression (\ref{eq:Psi}) we obtain a holomorphic map 
\begin{equation}
\label{eq:Phi}
	\Phi(\zeta,x,z) =
	\phi^1_{\zeta_{1,1}g_{1,1}(x)} \circ \cdots  \circ \phi^m_{\zeta_{l,m}g_{l,m}(x)} (z)
	\in A, 
\end{equation}
defined for $z\in A$, $x \in M$, and $\zeta$ in a neighborhood $\wt U\subset \C^{lm}$ of the origin. (This neighborhood can be chosen independent of $z$ for all points $z$ in any compact subset of $A$.)  Setting $z=f(x)$ also gives the holomorphic map
\begin{equation}
\label{eq:Phif}
	(\zeta,x) \longmapsto \Phi_f(\zeta,x)= \Phi(\zeta,x,f(x)) \in A\setminus\{0\}
\end{equation}
defined for  $\zeta \in \C^{lm}$ near the origin. Note that $\Phi_f(0,\cdotp)=f$.

If the approximations of $h_{i,j}$'s by $g_{i,j}$'s are close enough, then the corresponding period map $\zeta\mapsto \cP(\Phi_f(\zeta,\cdotp)) \in (\C^n)^l$ still has maximal rank at $\zeta=0$. Furthermore, by varying $f$ locally (keeping the functions $g_{i,j}$ fixed) we obtain a holomorphic family of maps $f\mapsto \Phi_f$ with the desired properties.

This proves Lemma \ref{lem:deformation} and hence part (b) of Theorem \ref{th:local}.
\qed

\begin{remark}
\label{rem:split}
The construction in \cite{FF:manifolds} shows that, if $X$ is a complex submanifold of a complex manifold $Y$, then $\Ascr(M,X)$ is a split Banach submanifold of $\Ascr(M,Y)$ (see \cite[p.\ 490]{Lempert} for this notion). Thus $\Ascr(M,A\setminus \{0\})$ is a (nonclosed) split Banach submanifold of the Banach space $\Ascr(M,\C^n)$. Furthermore, Lemma \ref{lem:deformation} shows that $\Ascr^*_{reg}(M,A\setminus\{0\})$ is a (nonclosed) split Banach submanifold of $\Ascr(M,A\setminus \{0\})$. By integration (\ref{eq:integral}) we get the corresponding statements for the inclusion $\IA\hra \Ascr^1(M,\C^n)$ at any nondegenerate $A$-immersion $F\in \IA$.
\qed\end{remark}

\noindent\it Proof of part (c). \rm
Fix $F\in \IA$. By part (a) we may assume that $F$ is nondegenerate. Write $dF=f\theta$, and let $\Phi_f$ be the deformation map furnished by Lemma \ref{lem:deformation}. We can approximate $f$ uniformly on $M$ by holomorphic maps $\tilde  f\colon V\to A\setminus \{0\}$, defined on small open neighborhoods $V=V_{\tilde f}$ of $M$ in a larger Riemann surface $R$. The associated deformation map $\Phi_{\tilde f}$ is then defined and holomorphic in a neighborhood $\wt U\times \wt V \subset \C^N\times R$ of $\{0\}\times M$. (Recall that the functions $g_{i,k}$ used in the construction of $\Phi_f$ are holomorphic in a neighborhood of $M$, and we may use the same functions for $\Phi_{\tilde f}$.) If $\tilde f$ is sufficiently close to $f$ on $M$, then the domain and the range of the period map $\cP(\Phi_{\tilde f})$ are so close to those of $\cP(\Phi_f)$ that the range of $\cP(\Phi_{\tilde f})$ contains the origin in $\C^{ln}$. This means that $\tilde f$ can be approximated by a map $h \in \Oscr(M,A\setminus\{0\})$ with vanishing periods. The integral $H(x)=F(p)+\int_p^x h\, d\theta$ is then a holomorphic $A$-immersion from a neighborhood of $M$ to $\C^n$ which approximates $F$ in ${\mathscr C}^1(M,\C^n)$.
\qed

%
\section{Approximation by directed embeddings}
\label{sec:desing}
In this section we prove Theorems \ref{th:desing} and \ref{th:desing2} concerning the approximation of $A$-immersions by $A$-embeddings. 

\smallskip
\noindent \it Proof of Theorem \ref{th:desing}. \rm 
We may assume that $M$ is a smoothly bounded domain in an open Riemann surface $R$. Let $F\colon M\to \C^n$ be an $A$-immersion. In view of Theorem \ref{th:local} we may assume that $F$ is holomorphic in a neighborhood of $M$ in $R$, and is nondegenerate in the sense of Def.\ \ref{def:nondegenerate}. We associate to $F$ the {\em difference map}
\[
	\delta F\colon M\times M\to \C^n,\qquad \delta F(x,y)=F(y)-F(x).
\]
Clearly $F$ is injective if and only if $(\delta F)^{-1}(0)= D_M=\{(x,x)\colon x\in M\}$, the diagonal of $M\times M$.

Since $F$ is an immersion, it is locally injective, and hence there is an open neighborhood $U\subset M\times M$ of $D_M$ such that $\delta F$ does not assume the value $0\in \C^n$ in $\overline U\setminus D_M$. To prove the theorem, it suffices to find arbitrarily close to $F$ another $A$-immersion $\wt F\colon M\to\C^n$ whose difference map $\delta \wt F$, restricted to $M\times M\setminus U$, is transverse to the origin $0\in \C^n$. Indeed, since $\dim_\C M\times M=2<n$, this will imply that $\delta\wt F$ does not assume the value zero in $M\times M\setminus U$, so $\wt F(x)\ne \wt F(y)$ if $(x,y)\in M\times M\setminus U$. If on the other hand $(x,y)\in U \setminus D_M$, then $\wt F(x)\ne \wt F(y)$ provided that $\wt F$ is sufficiently close to $F$. Thus $\wt F$ is an injective immersion, hence an embedding. 

A map $\wt F$ with these properties will be constructed by the standard transversality argument (see Abraham \cite{Abraham} or \cite[Sec.\ 7.8]{F:book}). We need to find a neighborhood $\Omega \subset \C^N$ of the origin in a complex Euclidean space and a holomorphic map $H\colon \Omega \times M \to \C^n$ such that $H(0,\cdotp)=F$ and the difference map $\delta H\colon \Omega \times M\times M \to \C^n$, defined by 
\begin{equation}
\label{eq:difference}
	\delta H(\zeta,x,y) = H(\zeta,y)-H(\zeta,x), \qquad \zeta\in \Omega, \ \ x,y\in M, 
\end{equation}
is a {\em submersive family of maps}, meaning that its partial differential 
\begin{equation}
\label{eq:pd}
	d_\zeta|_{\zeta=0} \delta H(\zeta,x,y) \colon \C^N \to \C^n
\end{equation}
is surjective for any $(x,y)\in M\times M\setminus U$. By openess of this condition and compactness of $M\times M \setminus U$ it follows that the partial differential $d_\zeta \delta H$ is surjective for all $\zeta$ in a neighborhood $\Omega'\subset \Omega$ of the origin in $\C^N$. Hence the map $\delta H \colon M\times M\setminus U\to\C^n$ is transverse to any submanifold of $\C^n$, in particular, to the origin $0\in \C^n$. The standard argument then shows that for a generic member $H(\zeta, \cdotp)\colon M \to \C^n$ of this family, the difference map $\delta H(\zeta,\cdotp)$ is also transverse to $0\in\C^n$ on $M\times M\setminus U$. By choosing the point $\zeta$ sufficiently close to $0$ we thus obtain a desired $A$-embedding $\wt F=H(\zeta,\cdotp)$.  

We now construct a deformation family $H$ with the above properties. 

Fix a nowhere vanishing holomorphic $1$-form $\theta$ on $R$ and write $dF=f\theta$, with $f\colon M\to A\setminus \{0\}$ a holomorphic map. Pick a neighborhood $U\subset M\times M$ of $D_M$ such that $\overline U\cap (\delta F)^{-1}(0)=D_M$. 

\begin{lemma}
\label{lem:pq}
{\rm (Notation as above.)} 
For any $(p,q)\in M\times M\setminus D_M$ there exists a deformation family $H=H^{(p,q)}(\zeta,\cdotp)$ as above, with $\zeta\in \C^n$, such that the differential $d_\zeta|_{\zeta=0} \delta H(\zeta,p,q) \colon \C^n \to \C^n$ is an isomorphism. 
\end{lemma}

Suppose that the lemma holds. Clearly $H$ satisfies the same property for all pairs $(p',q')\in M\times M$ close to $(p,q)$. Since $M\times M\setminus U$ is compact, it is covered by finitely many such neighborhoods. The superposition of the corresponding deformation families will yield a deformation family $H$ for which the differential (\ref{eq:pd}) is surjective for any $(x,y)\in M\times M\setminus U$. 

\begin{proof}
Let $\Lambda \subset M$ be a smooth embedded arc connecting $p$ to $q$. Pick a point $p_0\in M\setminus\Lambda$ and closed loops $C_1,\ldots, C_l \subset M\setminus\Lambda$ based at $p_0$ and forming a basis of $H_1(M;\Z)$. Set $C=\bigcup_{j=1}^l C_j$. Let $\gamma_j\colon [0,1]\to C_j$ ($j=1,\ldots,l$) and $\lambda\colon [0,1]\to \Lambda$ be smooth parametrizations of the respective curves.  

Since $F$ is nondegenerate (Def.\ \ref{def:nondegenerate}), there exist tangential holomorphic vector fields $V_1,\ldots,V_n$ on $A$ and points $x_1,\ldots, x_n\in \Lambda\setminus \{p,q\}$ such that, setting $z_i=f(x_i)\in A$, the vectors $V_i(z_i)$ for $i=1,\ldots,n$ span $\C^n$. (Such points exist since a nontrivial arc $\Lambda$ in $M$ is a determining set for holomorphic functions on $M$, and hence the tangent spaces $T_{f(x)} A$ over all points $x\in \Lambda$ have the same span as the tangent spaces $T_{f(x)} A$ over all points $x\in M$. Of course one could also move the curve $\Lambda$ a little to insure this property.) Let $t_i\in (0,1)$ be such that $\lambda(t_i)=x_i$. Let $\phi^i_t$ denote the flow of $V_i$. Choose smooth functions $h_i\colon C\cup \Lambda \to \R_+$ $(i=1,\ldots,n)$ that vanish at the endoints $p,q$ of $\Lambda$ and on the curves $C$; their value on $\Lambda$ will be chosen later. Let $\zeta=(\zeta_1,\ldots,\zeta_n) \in \C^n$. As in the proof of Lemma \ref{lem:deformation} we consider the map 
\[
	\psi(\zeta,x)=\phi^1_{\zeta_{1}h_{1}(x)} \circ \cdots  \circ \phi^n_{\zeta_{n}h_{n}(x)} (f(x)) 
	\in A\setminus\{0\},\quad x\in C\cup\Lambda,
\] 
which is holomorphic in $\zeta =(\zeta_1,\ldots,\zeta_n)\in\C^n$ near the origin. Note that $\psi(0,\cdotp)=f$ and $\psi(\zeta,x)=f(x)$ if $x\in C$ (since $h_i=0$ on $C$). We have 
\[
	\frac{\di \psi(\zeta,x)}{\di \zeta_{i}}\bigg|_{\zeta=0} = h_{i}(x)\, V_i(f(x)),\quad i=1,\ldots,n.
\]
By choosing the function $h_i$ to have its support concentrated near the point $x_i=\lambda(t_i) \in \Lambda$, we can arrange that for all $i=1,\ldots,n$ we have
\[
	\int_0^1 h_{i}(\lambda(t))\, V_i(f(\lambda(t))) \, \theta(\lambda(t),\dot\lambda(t)) \,dt
	\approx V_i(z_i) \, \theta(\lambda(t_i),\dot{\lambda}(t_i)) \in \C^n.
\]
Assuming that the approximations are close enough, the vectors on the left hand side above form a basis of $\C^n$. 

Fix a number $\epsilon>0$; its precise value will be chosen later. We apply Mergelyan's theorem to find holomorphic functions $g_i\colon M\to\C$ such that 
\[
	\sup_{C\cup\Lambda} |g_i-h_i| <\epsilon \quad \text{for}\ \ i=1,\ldots,n. 
\]
In the analogy with (\ref{eq:Phi}) and (\ref{eq:Phif}) we define holomorphic maps
\begin{eqnarray}
\label{eq:Psi2}
	\Psi(\zeta,x,z) &=&
	\phi^1_{\zeta_{1}g_{1}(x)} \circ \cdots  \circ \phi^n_{\zeta_{n}g_{n}(x)} (z) \in A,  \cr
	\Psi_f(\zeta,x) &=& \Psi(\zeta,x,f(x)) \in A,
\end{eqnarray}
where $x\in M$, $z\in A$, and $\zeta$ is near the origin in $\C^n$. Note that $\Psi_f(0,\cdotp)=f$.
If the approximations of $h_i$ by $g_i$ are close enough, then the vectors
\begin{eqnarray}
\label{eq:derivatives}
	\frac{\di}{\di \zeta_{i}}\bigg|_{\zeta=0} \int_0^1 \Psi_f(\zeta,\lambda(t)) 
	\,\theta(\lambda(t),\dot\lambda(t)) \,dt  \qquad \qquad \qquad\qquad \cr
	\qquad \qquad \qquad 
	= \int_0^1 g_{i}(\lambda(t))\, V_i(f(\lambda(t))) \, \theta(\lambda(t),\dot\lambda(t)) \,dt \in\C^n
\end{eqnarray}
are still close enough to the vectors $V_i(z_i) \, \theta(\lambda(t_i),\dot{\lambda}(t_i))$ for $i=1,\ldots,n$ so that they are linearly independent. 

The $\C^n$-valued $1$-form $\Psi_f(\zeta,\cdotp) \,\theta$ on $M$ need not be exact. We shall now correct its periods to zero by using the tools from Sec.\ \ref{sec:local}. 

From the Taylor expansion of the flow of a vector field we see that
\[
	\Psi_f(\zeta,x)=f(x)+\sum_{i=1}^n \zeta_i g_i(x) V_i(f(x)) + o(|\zeta|).
\]
Since $|g|<\epsilon$ on $C$, the periods over the loops $C_j$ can be estimated by 
\begin{equation}
\label{eq:estimate-periods}
	\left| \int_{C_j} \Psi_f(\zeta,\cdotp)\, \theta \right| \le \eta_0\epsilon |\zeta|
\end{equation}
for some constant $\eta_0>0$ and for sufficiently small $|\zeta|$. 

Lemma \ref{lem:deformation} gives holomorphic maps $\Phi(\wt \zeta,x,z)$ and $\Phi_f(\wt\zeta,x)=\Phi(\wt\zeta,x,f(x))$ (see (\ref{eq:Phi}) and (\ref{eq:Phif})), with the parameter $\wt \zeta$ near $0\in\C^{\wt N}$ for some $\wt N\in\N$ and $x\in M$, such that $\Phi(0,x,z)=z$ and the differential of the associated period map $\wt \zeta \mapsto \cP(\Phi_f(\wt\zeta,\cdotp)) \in \C^{ln}$ (see (\ref{eq:P})) at the point $\wt\zeta=0$ has maximal rank equal to $ln$. The same is true if we let the map 
$f\in \Ascr(M,A\setminus\{0\})$ vary locally near the given initial map. In particular, we can replace $f$ by the deformation family $\Psi_f(\zeta,\cdotp)$ and consider the composed map 
\[
	\C^{\wt N}\times \C^n\times M  \ni (\wt \zeta,\zeta,x) \longmapsto \Phi(\wt\zeta,x,\Psi_f(\zeta,x)) 
	\in A\setminus \{0\}
\]
which is defined and holomorphic for $(\wt \zeta,\zeta)$ near the origin in $\C^{\wt N}\times \C^n$ and for $x\in M$. The implicit function theorem furnishes a holomorphic map $\wt \zeta=\rho(\zeta)$ near $\zeta=0\in\C^n$, with $\rho(0)=0\in\C^{\wt N}$, such that the $\C^n$-valued 1-form on $M$, defined by
\[
	\Theta_f(\zeta,x,v)= \Phi(\rho(\zeta),x,\Psi_f(\zeta,x))\, \theta(x,v), \quad x\in M,\ v\in T_x M,
\]
has vanishing periods over the curves $C_j$ for every fixed $\zeta \in\C^n$ near $0$. (The map $\rho=(\rho_1,\ldots,\rho_n)$ also depends on $f$, but we shall suppress this dependence in our notation.) It follows that the integral
\begin{equation}
\label{eq:H}
	H_F(\zeta,x) =  F(p_0)+\int_{p_0}^x \Theta_f(\zeta,\cdotp,\cdotp) =
	F(p_0) + \int_0^1 \Theta_f(\zeta,\gamma(t),\dot\gamma(t)) 
\end{equation}
is independent of the choice of the path $\gamma$ from $p_0$ to $x\in M$. Clearly $H_F(0,\cdotp)=F$, and $H_F(\zeta,\cdotp) \colon M\to\C^n$ is an $A$-immersion for every $\zeta\in\C^n$ sufficiently close to $0$. Furthermore, in view of  (\ref{eq:estimate-periods}) we have the estimate
\[
	|\rho(\zeta)| \le \eta_1\epsilon |\zeta|
\]
for some $\eta_1>0$. The map $\Phi(\wt \zeta,x,z)$ is of the form (\ref{eq:Phi}), i.e., it is obtained by composing the flows of certain holomorphic vector fields $W_j$ on $A$ for the times $\wt \zeta_j \wt g_j(x)$, where $\wt g_j\in\Oscr(M)$. The Taylor exansion of the flow, together with the above estimate on $\rho(\zeta)$, give 
\begin{eqnarray*}
	\left| \Phi(\rho(\zeta),x,\Psi_f(\zeta,x)) - \Psi_f(\zeta,x) \right| &=& \left| \sum \rho_j(\zeta) \wt g_j(x)   W_j(\Psi_f(\zeta,x)) + o(|\zeta|) \right| \cr
	&\le& \eta_2\epsilon |\zeta| \cr
\end{eqnarray*}
for some $\eta_2>0$ and for all $x\in M$ and all $\zeta$ near the origin in $\C^n$. By applying this estimate on the curve $\Lambda$ (with the endpoints $p$ and $q$) we get 
\[
	\left| \int_0^1 \Theta_f(\zeta,\lambda(t),\dot\lambda(t)) - 
	\int_0^1 \Psi_f(\zeta,\lambda(t)) \,\theta(\lambda(t),\dot\lambda(t)) \,dt \right| \le \eta_3 \epsilon |\zeta|
\]
for some $\eta_3>0$. If $\epsilon>0$ is chosen small enough, it follows that the derivatives
\[
	\frac{\di}{\di \zeta_{i}}\bigg|_{\zeta=0} \int_0^1 \Theta_f(\zeta,\lambda(t),\dot\lambda(t)) \in\C^n,
	\qquad i=1,\ldots,n,   
\]
are so close to the vectors (\ref{eq:derivatives}) that they are $\C$-linearly independent. In view of (\ref{eq:H}) we have
\[
	\int_0^1 \Theta_f(\zeta,\lambda(t),\dot\lambda(t)) = H_F(\zeta,q)-H_F(\zeta,p)=\delta H_F(\zeta,p,q),
\]
where $\delta H_F$ is the difference map (\ref{eq:difference}). Hence the above says that the partial differential 
\[
	\frac{\di}{\di\zeta}\big|_{\zeta=0} \delta H_F(\zeta,p,q) \colon \C^n\to\C^n
\]
is an isomorphism. This proves Lemma \ref{lem:pq}.
\end{proof}

The family $H_F$ obtained above is holomorphically dependent also on $F$ in a neighborhood of a given initial $A$-immersion $F_0$. In particular, if $F(\eta,\cdotp)\colon M\to \C^n$ is a family of holomorphic $A$-immersions depending holomorphically on a complex parameter $\eta$, then $H_{F(\eta,\cdotp)}(\zeta,\cdotp)$ depends holomorphically on $(\zeta,\eta)$. This allows us to compose any finite number of such deformation families. We explain this operation for two families. Suppose that $H=H_F(\zeta,\cdotp)$ and $G=G_F(\eta,\cdotp)$ are  deformation families with $H_F(0,\cdotp)=G_F(0,\cdotp)=F$. We define the composed deformation family by 
\[
	(H \sharp G)_F(\zeta,\eta,x)=G_{H_F(\zeta,\cdotp)} (\eta,x),\quad x\in M.
\]
Clearly we have
\[
	(H\sharp G)_F(0,\eta,\cdotp)=G_F(\eta,\cdotp), \quad (H\sharp G)_F(\zeta,0,\cdotp)=H_F(\zeta,\cdotp).
\]
The operation $\sharp$ extends by induction to finitely many factors; it is associative, but not commutative. (This operation is similar to the composition of sprays that was introduced by Gromov \cite{Gromov:Oka}; see also \cite[p.\ 246]{F:book}.)

We can now complete the proof of Theorem \ref{th:desing}. The above construction gives a finite open covering $\cU=\{U_i\}_{i=1}^m$ of the compact set $M \times M\setminus U$ and deformation families $H^i=H^i(\zeta^i,\cdotp)\colon M\to \C^n$, with  $H^i(0,\cdotp)=F$, where $\zeta^i=(\zeta^i_1,\ldots,\zeta^i_{k_i}) \in \Omega_i\subset \C^{k_i}$, so that the difference map $\delta H^i(\zeta^i,p,q)$ is submersive at $\zeta^i=0$ for all $(p,q)\in U_i$. By taking $\zeta=(\zeta^1,\ldots,\zeta^m)\in \C^N$, with $N=\sum_{i=1}^m k_i$, and setting
\[
	H(\zeta,x) = (H^1 \sharp H^2 \sharp \cdots \sharp H^m)(\zeta^1,\ldots,\zeta^m,x)
\]
we obtain a deformation family such that $H(0,\cdotp)=F$ and $\delta H$ is submersive everywhere on $M\times M\setminus U$ for all $\zeta\in \C^N$ sufficiently close to the origin. This completes the proof of Theorem \ref{th:desing}. 
\qed 
\smallskip

\noindent \it Proof of Theorem \ref{th:desing2}. \rm
Let $F\colon M\to \C^n$ be an $A$-immersion of an open Riemann surface to $\C^n$. Fix a number $\epsilon>0$ and a compact set $K\subset M$. Write $F_0=F$ and $\epsilon_0=\epsilon$. Choose an exhaustion of $M$ by an increasing sequence $M_0\subset M_1\subset \cdots \bigcup_{j=0}^\infty M_j = M$ of smoothly bounded compact domains such that every $M_j$ is holomorphically convex in $M$ and $K\subset M_0$. By Theorem \ref{th:desing} we can find an $A$-embedding $\wt F_1\colon M_1\to \C^n$ such that 
\[
	||\wt F_1-F_0||_{M_1} := \sup_{x\in M_1} |\wt F_1(x)-F_0(x)|<\epsilon/4. 
\]

Since $A\setminus\{0\}$ is assumed to be an Oka manifold, Corollary \ref{cor:Runge} gives an $A$-immersion $F_1 \colon M\to \C^n$ such that $F_1$ is an embedding on $M_1$, $||F_1-\wt F_1||_{M_1}<\epsilon/4$, and hence $||F_1-F_0||_{M_1}<\epsilon/2$. (Although Corollary \ref{cor:Runge} is proved in Sec.\ \ref{sec:Oka} below, its proof is independent of the results in this section.) 

Pick a number $\epsilon_1$ with $0<\epsilon_1<\epsilon/4$ such that every immersion $G\colon M\to \C^n$ satisfying $||G-F_1||_{M_1} <\epsilon_1$ is an embedding on $M_0$. (Such $\epsilon_1$ exists by the Cauchy estimates.) By applying the above argument to $F_1$ we find an $A$-immersion $F_2\colon M\to \C^n$ which is an embedding on $M_2$ and satisfies $||F_2-F_1||_{M_2} <\epsilon_1/2$. Pick a number $\epsilon_2$ with $0<\epsilon_2<\epsilon_1/4$ such that every holomorphic map $G\colon M\to \C^n$ satisfying $||G-F_2||_{M_2} <\epsilon_2$ is an embedding on $M_1$. Continuing inductively, we find a sequence of $A$-immersions $F_j\colon M\to\C^n$ and and a sequence of numbers $\epsilon_j>0$ such that the following hold for every $j=1,2\ldots$:
\begin{itemize}
\item[\rm (a)]  $F_{j}$ is an embedding on $M_{j}$, 
\item[\rm (b)]  $||F_{j}-F_{j-1}||_{M_{j}} < \epsilon_{j-1}/2$, 
\item[\rm (c)]  $0<\epsilon_{j} <\epsilon_{j-1}/4$,  and 
\item[\rm (d)]  every holomorphic map $G\colon M\to \C^n$ satisfying $||G-F_j||_{M_{j}} <\epsilon_{j}$ is an embedding on $M_{j-1}$.      
\end{itemize}
Property (c) implies that $\sum_{k=j+1}^\infty \epsilon_k<\epsilon_j/2$ for every $j=0,1,\ldots$. By property (b) we see that the limit $\wt F=\lim_{j\to\infty} F_j\colon M\to \C^n$ exists and satisfies $||\wt F-F||_{K}<\epsilon$. Furthermore, we have 
\[
	||\wt F-F_j||_{M_{j}} \le \sum_{k=j}^\infty ||F_{k+1}-F_k||_{M_{j}}
	< \frac{\epsilon_j}{2} + \sum_{k=j+1}^\infty \epsilon_k < \epsilon_j,
\]
and hence $\wt F$ restricted to $M_{j-1}$ is an $A$-embedding by properties (a) and (d). Since this holds for every $j$, we see that $\wt F\colon M\to \C^n$ is an $A$-embedding.
\qed

Theorems \ref{th:desing} and \ref{th:desing2} allow us to extend the known existence theorems for immersed null curves to the embedded case. The following corollary is of particular interest.

\begin{corollary}\label{co:CY}
Let $N$ be an orientable noncompact smooth real surface without boundary, and let $\Omega\subset\C^3$ be a convex domain. Then there exists a complex structure $J$ on $N$ such  that $(N,J)$ embeds as a complete proper null curve in $\Omega$.
\end{corollary}

If one takes $\Omega$ to be a bounded convex domain contained in $\C^3\setminus\{z_3 =0\}$, then the correspondence $\cT$ (\ref{eq:T}) applies and embeds $(N,J)$ as a complete bounded null curve in $SL_2(\C)\setminus\{z_{11}=0\}$.

\begin{proof}
Let $M$ be an open Riemann surface diffeomorphic to $N$. It is shown in \cite{AL1} that there exist an increasing sequence of smoothly bounded Runge domains $M_1\subset M_2\subset\ldots \subset M$ and null curves $F_j\colon M_j\to\C^3$, $j\in\N$, such that the limit map $F=\lim_{j\to\infty} F_j\colon\bigcup_{j\in\N} M_j\to \C^3$  exists and is a complete null curve mapping the domain $D=\bigcup_{j\in\N} M_j \subset M$ properly into $\Omega$. Furthermore, we can arrange that $D$ is homeomorphic (and hence diffeomorphic) to $M$, and hence to $N$. Let $J$ be the complex structure on $N$ obtained from the complex structure on $D$ via this diffeomorphism.

Now Theorem \ref{th:desing} insures that such $F_j$'s can be chosen to be embeddings. If $F_j$ is close enough to $F_{j-1}$ on $M_{j-1}$ for all $j>1$ (see \cite[Lemma 3]{AL1}), then the limit null curve $F\colon D\to \Omega$ is embedded as well.
\end{proof}

Corollary \ref{co:CY} is motivated by the question whether there exist complete bounded minimal surfaces in $\R^3$; a classical problem in the theory of minimal surfaces, known as the {\em Calabi-Yau problem}. The answer to this question strongly depends on whether self-intersections are allowed or not. In the immersed case, such surfaces exist and may have arbitrary topological type \cite{Na,FMM,AL1}. On the other hand, complete embedded minimal surfaces with finite genus and countably many ends are necessarily proper in $\R^3$ \cite{CM,MPR}, hence unbounded. The general problem remains open for embedded surfaces. The corresponding question for immersed null curves in $\C^3$ was answered affirmatively in \cite{AL1}; however, the methods in \cite{AL1} do not enable one to avoid the self-intersections, nor to control the complex structure on the curve. The former question (see \cite[Problem 2]{MUY1}) is solved by Corollary \ref{co:CY}. For the latter, a technique for constructing complete bounded complex curves in $\C^2$ that are normalized by any given bordered Riemann surface has been  developed recently in \cite{AF}. The analogous problems for minimal surfaces in $\R^3$ and null curves in $\C^3$ remain open (see \cite[Question 2]{AF}).


%
\section{The Oka principle and Mergelyan's theorem for $A$-immersions}
\label{sec:Oka}

In this section we prove Theorem \ref{th:Oka} -- the Oka principle for $A$-immersions. The same proof also gives the Mergelyan approximation theorem for $A$-immer\-sions; see Theorem \ref{th:Mergelyan} below.

We begin by introducing a suitable type of sets  for the Mergelyan theorem. We shall not strive for the most general possible situation; the type of sets in the following definition suffice in most geometric applications.

\begin{definition}
\label{def:admissible}
A compact subset $S$ of an open Riemann surface $M$ is said to be {\em admissible} if $S=K\cup C$, where $K=\bigcup \ol D_j$ is a union of finitely many pairwise disjoint, compact, smoothly bounded domains $\ol D_j$ in $M$ and $C=\bigcup C_i$ is a union of finitely many pairwise disjoint smooth arcs or closed curves that intersect $K$ only in their endpoints (or not at all), and such that their intersections with the boundary $bK$ are transverse. 
\end{definition}

An admissible set $S\subset M$ is Runge in $M$ if and only if the inclusion map $S\hra M$ induces an injective homomorphism $H_1(S;\Z)\hra H_1(M;\Z)$ of the first homology groups. If this holds, then we have the classical Mergelyan approximation theorem: Every continuous function $f\colon S\to \C$ that is holomorphic in the interior $\mathring K$ of $K$ can be approximated, uniformly on $S$, by functions holomorphic on $M$. If in addition $f$ is of class ${\mathscr C}^1$ on $S$, then the approximation can be made in the ${\mathscr C}^1(S)$ topology. 

The notion of an $A$-immersion extends in an obvious way to maps $F\colon S\to \C^n$ of class ${\mathscr C}^1(S)$. On the set $K$ this is the standard notion, while on the curves $C$ we ask that the derivative $F'(t)$ with respect to any local real parameter $t$ on $C$ belongs to $A\setminus \{0\}$.

\begin{theorem}
\label{th:Mergelyan}
{\rm (Mergelyan's theorem for $A$-immersions.)}
Let $A\subset \C^n$ be a closed irreducible conical subvariety which is smooth away from $0$. Assume that $M$ is an open Riemann surface and that $S=K\cup C$ is a compact admissible set in $M$ (see Def.\ \ref{def:admissible}). Then the following hold:
\begin{itemize}
\item[\rm (a)] Every $A$-immersion $S\to \C^n$ can be approximated in the ${\mathscr C}^1(S)$ topology by $A$-immersions $U\to \C^n$ in open neighborhoods of $S$ in $M$. 
\item[\rm (b)] 
If in addition $S$ is Runge in $M$ and $A\setminus \{0\}$ is an Oka manifold (Def.\ \ref{def:Oka}), then every $A$-immersion $S\to \C^n$ can be approximated in the ${\mathscr C}^1(S)$ topology by $A$-immersions $M\to \C^n$.
\end{itemize}
\end{theorem}

Theorem \ref{th:Mergelyan} also holds when $M$ is a compact bordered Riemann surface since every such is a smoothly bounded domain in an open Riemann surface.

We shall need the following lemma which is analogous to Gromov's convex integration lemma (see \cite{Gromov:convex} or \cite{EM}).

\begin{lemma}
\label{lem:convex}
Let $A\subset \C^n$ be an irreducible conical subvariety which is not contained in any hypersurface. Given continuous maps $h\colon [0,1]\to A\setminus \{0\}$ and $g\colon [0,1]\to \C \setminus \{0\}$, a vector $v\in \C^n$, and a number $\epsilon>0$, there exists a homotopy $h_s\colon [0,1]\to A\setminus \{0\}$ $(0\le s\le 1)$ such that $h_0=h$, the homotopy is fixed near the endpoints $0$ and $1$, and we have
\begin{equation}
\label{eq:intv}
	\left| \int_0^1 h_1(t)g(t)\, dt - v\right| <\epsilon.
\end{equation}
\end{lemma}

\begin{proof} 
Replacing the map $h$ by $hg$ (which also has range in $A\setminus\{0\}$) we reduce to the case when $g\equiv 1$. Since $A$ is conical and its convex hull equals $\C^n$ (see Lemma \ref{lem:convexhull}), we can find an integer $N$ and vectors $v_1,\ldots,v_N\in A \setminus\{0\}$ such that $\frac{1}{N} \sum_{j=1}^N v_j = v$. Set $v_0=h(0)$ and $v_{N+1}=h(1)$. Choose a small number $\delta>0$ and let $I_j\subset [0,1]$ be the pairwise disjoint segments
\[
	I_j=\left[\frac{j-1+\delta}{N}, \frac{j-\delta}{N}\right],\qquad j=1,\ldots,N. 
\]
Their complement $J=[0,1]\setminus\bigcup_{j=1}^N I_j$ has total length $2\delta$. Let $h_1\colon [0,1]\to A \setminus\{0\}$ be chosen such that $h_1=h$ near $0$ and $1$, and $h_1(t)=v_j$ for $t\in I_j$ ($j=1,\ldots,N$). On the remaining segments contained in $J$ we choose $h_1$ so that it is continuous and homotopic to $h$, and so that $|h_1|\le R$ for some constant $R>0$ independent of $\delta$. This is achieved by first going from $v_0=h(0)$ to $v_1$ along a path in $A\setminus\{0\}$ in time $[0,\delta/N]$, then staying at $v_1$ for time $t\in I_1$, then going from $v_1$ to $v_2$ along an arc in $A\setminus\{0\}$ in time $(1-\delta)/N \le t\le (1+\delta)/N$, then staying at $v_2$ for time $t\in I_2$, etc. By a suitable choice of the arcs connecting the consecutive points $v_j$, $v_{j+1}$, we can insure that the new path $h_1$ is homotopic to $h$ and that it remains in a fixed ball $\{z\in\C^n\colon |z|<R\}$ independent of $\delta$. We then have 
\[
	\int_0^1 h_1(t)\, dt =  \frac{1-2\delta}{N} \sum_{j=1}^N v_j  + \int_J h_1(t)\,dt 
	= (1-2\delta) v + \int_J h_1(t)\,dt.
\]
Choosing $\delta< \epsilon/4R$ we get 
\[
	\left| \int_0^1 h_1(t)\,dt - v\right| \le 2\delta|v| + \int_J |h_1(t)| \,dt  < 
	4\delta R < \epsilon.
\]
This proves Lemma \ref{lem:convex}.
\end{proof}

\begin{remark}
\label{rem:convex}
We expect that one can always reach the equality in (\ref{eq:intv}), but this will not be needed. The assumption that $A$ is conical was used to reduce to the case $g=1$. Lemma \ref{lem:convex} still holds without this assumption which is seen as follows. We subdivide $[0,1]$ into a large number $N$ of sufficiently small subintervals $I_j$ such that $g$ is very close to a constant $g_j$ on each of them. Then we repeat the above argument on $I_j$ to find $h$ so that $\int_{I_j} h(t)g_j \,dt \approx N^{-1} v$. Summing up, we get $\int_0^1 hg \approx v$.  
\qed \end{remark}

\smallskip
\noindent \it Proof of Theorems \ref{th:Mergelyan} and \ref{th:Oka}. 
\rm We begin by proving part (a) of Theorem \ref{th:Mergelyan}. We may assume that $S$ is connected since the same argument applies separately to each connected component. 

We begin by perturbing the given $A$-immersion $F\colon S\to \C^n$ so as to make it nondegenerate in the sense of Def.\ \ref{def:nondegenerate}; this can be done as in the proof of Theorem \ref{th:local}-(a). By the proof of Theorem \ref{th:local}-(c), 
but using Mergelyan approximation, we can approximate the map $f=dF/\theta\colon S\to A\setminus\{0\}$, uniformly on $S$, by a holomorphic map $\wt f\colon U\to A\setminus\{0\}$ on an open connected neighborhood $U\subset M$ of $S$ such that $\wt f\theta$ has vanishing periods over all nontrivial loops in $S$. We may assume that $S$ is a strong deformation retract of $U$. We then get an $A$-immersion $\wt F\colon U\to \C^n$ by setting $\wt F(x) = F(p) + \int_p^x \wt f\theta$ $(x\in U)$ for any chosen point $p\in S$. By the construction, $\wt F|_S$ approximates $F$ in ${\mathscr C}^1(S)$ since the integral from $p$ to any point $x\in S$ can be calculated over a path in $S$ and $|\wt f-f|$ is small on $S$. 

This proves part (a) of Theorem \ref{th:Mergelyan}.

We now turn to the proof of Theorem \ref{th:Mergelyan}-(b). At the same time we shall prove Theorem \ref{th:Oka} (the Oka principle for $A$-immersions). In fact, since every compact Runge set $K\subset M$ has a basis of compact smoothly bounded Runge neighborhoods, the only addition in Theorem \ref{th:Oka} over Theorem \ref{th:Mergelyan} is that one  can prescribe the homotopy class of $f=dF/\theta \colon M\to A\setminus\{0\}$. 

By the already proved part (a) we may fix an $A$-immersion $F_0\colon U\to \C^n$ from an open set $U\subset M$ containing $S$ such that $F_0$ is ${\mathscr C}^1$ close to $F$ on $S$. Write $dF_0=f_0\theta$ where $f_0\colon U\to A\setminus\{0\}$. After shrinking $U$ around $S$ we may assume that $f_0$ extends to a contiuous map $f_0\colon M\to A\setminus \{0\}$.

Since the set $S$ is Runge in $M$, there exists a smooth strongly subharmonic Morse exhaustion function $\tau\colon M\to \R$ with nondegenerate (Morse) critical points such that $S\subset \{\tau<0\}$ and $\{\tau\le 0\}\subset U$. We may assume that $0$ is a regular value of $\tau$, so $D_0=\{\tau \le 0\}$ is a smoothly bounded compact domain.  
Let $p_1,p_2,\ldots$ be the critical points of $\tau$ in $M\setminus D_0$. We may assume that $0<\tau(p_1)<\tau(p_2)<\cdots$. 
Choose a strictly increasing divergent sequence $a_1<a_2<a_3<\cdots$ such that $\tau(p_j)< a_j< \tau(p_{j+1})$ holds for $j=1,2,\ldots$. If there are only finitely many $p_j$'s, then we choose the remainder of the sequence $a_j$ arbitrarily. 
Set $D_j=\{\tau\le a_j\}$. To conclude the proof we shall inductively construct a sequence of $A$-immersions 
\begin{equation}\label{eq:FjDj}
F_j\colon D_j \to \C^n,\quad j=1,2,\ldots
\end{equation}
such that $||F_j-F_{j-1}||_{D_{j-1}} <\epsilon_j$ for a certain sequence $\epsilon_j>0$ which decreases to zero sufficiently fast, and such that the map $dF_j/\theta=f_j\colon D_j\to A\setminus\{0\}$ is homotopic to $f_0|_{D_j}$ through maps 
$D_j\to A\setminus\{0\}$.

We begin the induction with the given $A$-immersion $F_0$. For the inductive step, assume that we have an  $A$-immersion $F_{j-1}$ for some $j>0$. The construction of $F_j$ from $F_{j-1}$ is made in two steps: the noncritical case when there is no change of topology of the sublevel set, and the critical case when the topology of the sublevel set changes at one point. If there are only finitely many $p_i$'s and there exists no value $\tau(p_i)$ between $a_{j-1}$ and $a_j$, then only the noncritical step is required; otherwise both steps are needed.

\smallskip
\noindent \it The noncritical case \rm is accomplished by the following lemma; it is at this point that the Oka property of $A\setminus \{0\}$ is invoked.

\begin{lemma}
\label{lem:noncritical}
Let $M$ be an open Riemann surface and $D\subset D'$ be compact domains with smooth boundaries in $M$. Assume that there is a smooth function $\sigma$ on a neighborhood $U \supset D'\setminus \mathring D$, with $d\sigma\ne 0$ on $U$, such that $D\cap U=\{\sigma\le a\}$ and $D'\cap U=\{\sigma\le b\}$ for some real numbers $a<b$. If $A\subset \C^n$ is as in Theorem \ref{th:Oka} (so $A\setminus \{0\}$ is an Oka manifold), then every $A$-immersion $F\colon D\to \C^n$ can be approximated, uniformly on $D$, by $A$-immersions $\wt F\colon D'\to \C^n$.
\end{lemma}

\begin{proof}
It suffices to consider the case when $D$ (and hence $D'$) is connected. By Theorem \ref{th:local} we may assume that $F\colon D\to \C^n$ is a nondegenerate $A$-immersion defined on an open neighborhood $V\subset M$ of $D$. Write $f=dF/\theta\colon V\to A\setminus\{0\}$. By Lemma \ref{lem:deformation} there exist an open ball $W$ around the origin in some $\C^N$ and a holomorphic map 
\[
	W \times V \ni (\zeta,x) \longmapsto \Phi(\zeta,x)\in A\setminus\{0\} 
\]
such that $\Phi(0,\cdotp)=f$ and the period map $\zeta\mapsto \cP(\Phi_f(\zeta,\cdotp)) \in (\C^n)^l$ (\ref{eq:P}) has maximal rank at $\zeta=0$. (The period map is now calculated on closed curves $C_1,\ldots,C_l$ in $D$ which form a basis of $H_1(D;\Z)$.)

Choose an open set $V'\subset M$ containing $D'$ such that $D'$ is a strong deformation retract of $V'$. Pick closed balls $B_0 \subset B \subset W$ around $0\in \C^N$. The compact set $B\times D$ is then Runge in the Stein manifold $W \times V'$ and is a strong deformation retract of $W \times V'$. Since $A\setminus \{0\}$ is an Oka manifold, we can apply Theorem \ref{th:OkaP} to approximate $\Phi$, uniformly on $B\times D$, by a holomorphic map $\Psi\colon W\times V'\to A\setminus\{0\}$. If the approximation is sufficiently close then the range of the period map $B_0 \ni \zeta\mapsto \cP(\Psi(\zeta,\cdotp)) \in (\C^n)^l$ still contains the origin, so there is a point $\zeta_0\in B_0$ such that the map $V' \ni x \mapsto \Psi(\zeta_0,x) \in A\setminus \{0\}$ has vanishing periods. Since $B_0$ can be chosen arbitrarily small, the integral of  $\Psi(\zeta_0,\cdotp)\,\theta$ is an $A$-immersion $\wt F \colon D'\to\C^n$ which approximates $F$ uniformly on $D$ as close as desired.
\end{proof}

\smallskip
\noindent \it The critical case. \rm  We now have compact domains $D\subset D'$ with smooth boundaries in $M$ and a smooth strongly subharmonic function $\tau$ on a neighborhood of $D'$, with $D=\{\tau\le a\}$ and $D'=\{\tau\le b\}$, such that $\tau$ has a unique (Morse) critical point $p$ on $D'\setminus \mathring D$. Let $F\colon D\to \C^n$ be an $A$-immersion and write $f=dF/\theta$.

Since $\tau$ is strongly subharmonic, the Morse index of $p$ is either 0 or 1. 

If the Morse index of $p$ equals 0, a new connected component of the sublevel set $\{\tau\le t\}$ appears at $p$ when $t$ passes the value $\tau(p)$. In this case we can extend $f$ by choosing an arbitrary holomorphic map from this new component to $A\setminus\{0\}$. This also provides an extension of $F$. 

If the Morse index of $p$ equals 1, then the change of topology of the sublevel set $\{\tau\le t\}$ at $p$ is described by attaching to $D=\{\tau\le a\}$ a smooth arc $C$ (the stable manifold of $p$ for the gradient flow of $\tau$). Let $q_1,q_2$ denote the endpoints of $C$. By applying Lemma \ref{lem:convex} and perturbing $F$ slightly on $D$ we can extend the map $f=dF/\theta\colon D\to A\setminus \{0\}$ smoothly to the arc $C$ so that the extended maps still has range in $A\setminus \{0\}$ and satisfies 
\begin{equation}
\label{eq:intC}
		\int_{C} f\theta = F(q_2)-F(q_1). 
\end{equation}
Indeed, Lemma \ref{lem:convex} provides an extension of $f$ to $C$ for which (\ref{eq:intC}) holds approximately. If the endpoints $q_1$ and $q_2$ of $C$ belong to different connected components of $D$, we adjust the value of $F$ on one of these two components by adding a suitable constant vector so as to make (\ref{eq:intC}) hold. If on the other hand the endpoints belong to the same component of $D$, then we perturb the difference $F(q_2)-F(q_1)$ by using Lemma \ref{lem:pq} to make (\ref{eq:intC}) hold. 

When proving Theorem \ref{th:Oka} (the Oka principle), we must also insure that the extended map $f\colon D\cup C \to A\setminus\{0\}$ constructed above is homotopic to the given continuous map $M\to A\setminus\{0\}$. This is easily achieved by a suitable choice of the connecting paths used in the proof of Lemma \ref{lem:convex}.

It follows from (\ref{eq:intC}) that the extended map $f$ integrates to an $A$-immer\-sion $F_0\colon D\cup C\to\C^n$ which agrees with $F$ on $D$. (If $D\cup C$ is disconnected, we integrate separately in each connected component.) By part (a) we can approximate $F_0$ by an $A$-immersion $F_1\colon U\to \C^n$ in a neighborhood $U\subset M$ of $D\cup C$. Now there is a smoothly bounded compact neighborhood $B\subset U$ of $D\cup C$ such that $D'$ is a noncritical extension of $B$ as in Lemma \ref{lem:noncritical}. Hence that lemma furnishes an $A$-immersion $\wt F \colon D'\to \C^n$ approximating $F_1$ on $B$. 

This completes the critical step, closes the induction, and concludes the construction of the sequence \eqref{eq:FjDj} with the desired properties. Indeed, just set $D=D_{j-1}$, $D'=D_j$, and $F=F_{j-1}$, and define $F_j:=\wt F$ given either by the critical or the noncritical step (depending on the topology of $D_j\setminus \mathring D_{j-1}$), such that $||F_j-F_{j-1}||_{D_{j-1}}<\epsilon_j$.

Finally, if the $\epsilon_j$'s are chosen small enough, then the limit $A$-immersion $\lim_{j\to\infty} F_j\colon M=\bigcup_j D_j\to\C^n$ is as close to $F$ in the ${\mathscr C}^1(S)$ topology as desired; recall that $S\subset D_0$ and that $F_0$ is ${\mathscr C}^1$ close to $F$ on $S$. This proves Theorem \ref{th:Mergelyan}-(b) and Theorem \ref{th:Oka}.
\qed
\smallskip

The following immediate corollary to Theorem \ref{th:Mergelyan}-(a) is obtained by using the correspondence $\cT$ (\ref{eq:T}).

\begin{corollary}\label{co:sl2c-mergelyan}
Let $A$, $M$, and $S$ be as in Theorem \ref{th:Mergelyan}-(a). Then every null curve $S\to SL_2(\C)\setminus\{z_{11}=0\}$ (in the sense described just above Theorem \ref{th:Mergelyan}) can be approximated in the ${\mathscr C}^1(S)$ topology by null curves $U\to SL_2(\C)$ in open neighborhoods $U$ of $S$ in $M$.
\end{corollary}

By a minor modification of the proof of Theorem \ref{th:Mergelyan} we now obtain Merge\-lyan's theorem for $A$-immersions with a fixed component function. This result will be used in an essential way in the construction of proper $A$-embeddings, given in the following section.

\begin{theorem}
\label{th:Mergelyan2}
{\rm (Mergelyan's theorem for $A$-immersions, second version.)}
Let $A\subset \C^n$ be as in Theorem \ref{th:Mergelyan}. Assume that $A\cap\{z_1=1\}$ is an Oka manifold  (Def.\ \ref{def:Oka}), and that the coordinate projection $\pi_1\colon A\to\C$ onto the $z_1$-axis admits a local holomorphic section $h$ near $z_1=0$ with $h(0)\ne 0$. Let $S$ be a compact admissible Runge set in an open Riemann surface $M$. Given an $A$-immersion $F=(F_1,F_2\ldots,F_n)\colon S\to\C^n$ such that $F_1$ extends to a nonconstant holomorphic function $F_1\colon M\to\C$, there exists for every $\epsilon>0$ a holomorphic $A$-immersion $\wt F= (F_1,\wt F_2\ldots,\wt F_n)\colon M\to\C^n$ such that 
$||\wt F_j-F_j||_{\mathscr C^1(S)} <\epsilon$ for $j=2,\ldots,n$.
\end{theorem}


The condition $\wt F_1=F_1$ in the above theorem is not a misprint: the first component of $F$ is kept fixed. This is the key addition over Theorem \ref{th:Mergelyan}.

\smallskip
\noindent
{\it Proof of Theorem \ref{th:Mergelyan2}.} 
Set $A':=A\cap \{z_1=1\}$. By using dilations we see that $A\setminus \{z_1=0\}$ is biholomorphic to $A' \times \C^*$ (and hence is Oka), and the projection $\pi_1\colon A\to\C$ is a trivial fiber bundle with Oka fiber $A'$ except over the origin $0\in \C$.  

Write $dF=f\theta$, where $f=(f_1,\ldots,f_n)=(f_1,f') \colon S\to A\setminus\{0\}$. We use the notation $f'=(f_2,\ldots,f_n)$. As in the proof of Theorem \ref{th:Mergelyan}-(a) we can approximate $F$ by an $A$-immersion in a neighborhood $U\subset M$ of $S$, without changing the first coordinate; the only difference is that we apply the correction of periods technique (see Sec.\ \ref{sec:local}) only to the component $f'$. To this end we use holomorphic vector fields on $A$ that are tangential to the fibers of the projection $\pi_1\colon A\to \C$.

Since the function $f_1=dF_1/\theta$ is holomorphic and nonconstant on $M$, its zero set $f_1^{-1}(0) = \{a_1,a_2,\ldots\}$ is discrete in $M$. The pullback $f_1^* \pi_1 \colon E=f^*A\to M$ of the projection $\pi_1\colon A\to \C$ to $M$ is a trivial holomorphic fiber bundle with fiber $A'$ over $M\setminus \{a_j\}$, but it may be singular over the points $a_j$. The map $f'\colon U\to \C^{n-1}$ satisfies $f'(x)\in \pi^{-1}_{1}(f_1(x))$ for $x\in U$, so $f'$ corresponds to a section of $E\to M$ over the set $U$. 

The problem now is to approximate $f'$, uniformly on a compact Runge neighborhood of $S$, by a section of $E\to M$ whose periods over all loops in $M$ are zero. Except for the period condition, a solution is provided by the Oka principle for sections of ramified holomorphic maps with Oka fibers (see \cite{FF:multivalued} or \cite[Sec.\ 6.13]{F:book}). The proof in our situation, when we must pay attention to the periods, is quite similar. We begin by choosing a local holomorphic solution in a small neighborhood of any point $a_j\in M\setminus S$ so that $f'(a_j)\ne 0$, and we add these neighborhoods to the domain of $f'$. 
We then follow the proof of Theorem \ref{th:Mergelyan} to enlarge the domain of holomorphicity of $f'$.

The noncritical case (see Lemma \ref{lem:noncritical}) amounts to approximating a holomorphic solution $f'$ on a compact smoothly bounded domain $D\subset M$ by a holomorphic solution on a larger such domain $D'\subset M$, assuming that there is no change of topology and that $D'\setminus \mathring D$ does not contain any of the points $a_j$. This is done by applying the Oka principle for maps to the Oka fiber $A'$ of $\pi_1\colon A\to\C$ over the set $\C^*$ where the bundle is trivial. 

In the critical case we attach a smooth arc $C$ to a domain $D\subset M$ such that $C$ does not contain any of the points $a_j$, and we extend $f'$ smoothly over $C$ so that the integral $\int_C f'\theta$ has the correct value (see the condition (\ref{eq:intC})). This is accomplished by a suitable analogue  of Lemma \ref{lem:convex} (cf.\ Remark \ref{rem:convex}). The extended map $f'$ integrates to a holomorphic map $F'\colon D\cup C\to\C^{n-1}$ such that $(F_1,F') \colon D\cup S\to \C^n$ is an $A$-immersion. The proof is finished as before by applying the noncritical case for another pair of domains. 
\qed

\begin{example}
Let $A\subset \C^3$ be the quadric variety (\ref{eq:null}) controlling null curves. Then $A\cap \{z_1=1\} = \{z_2^2 + z_3^2=-1\}$ is an embedded copy of the Oka manifold $\C^*=\C\setminus\{0\}$. (Besides $\C$, this is the only manifold for which Oka himself established the Oka principle in his pioneering paper \cite{Oka} from 1939.) In fact, any hyperplane section of $A$ which does not contain the origin is biholomorphic to $\C^*$. In this particular case, Theorem \ref{th:Mergelyan2} was proved by Alarc\'on and L\'opez \cite{AL1} by using the Weierstrass representation of null curves and the {\em L\'opez-Ros transformation}, a tool that was originally invented to prove a classification result for minimal surfaces in $\R^3$ \cite{LR}.  Neither of these tools is available in the general setting of the present paper.
\qed\end{example}

%

\section{Proper $A$-embeddings}
\label{sec:proper}

The aim of this final section is to prove the following existence result for proper directed embeddings of open Riemann surfaces in $\C^n$.

\begin{theorem}
\label{th:proper}
Let $A\subset \C^n$ $(n\ge 3)$ be a conical subvariety as in Theorem \ref{th:local}. Assume in addition that 
$A\setminus \{0\}$ is an Oka manifold (Def.\ \ref{def:Oka}), and that for $k\in\{1,2\}$ the hyperplane section $A\cap \{z_k=1\}$ is an Oka manifold and the coordinate projection $\pi_k\colon A\to\C$ onto the $z_k$-axis admits a local holomorphic section $h_k$ near $z_k=0$, with $h_k(0)\ne 0$.

Let $M$ be an open Riemann surface, and let $K\subset M$ be a compact Runge set. Then every $A$-immersion from an open neighborhood of $K$ in $M$ into $\C^n$ can be approximated in the ${\mathscr C}^1(K)$ topology by proper $A$-embeddings $M\to\C^n$. Furthermore, these $A$-embeddings can be chosen such that their first two coordinates determine a proper map of $M$ into $\C^2$.
\end{theorem}

Properly immersed null curves in $\C^3$ parametrized by any given open Riemann surface were constructed in \cite{AL2}. The methods developed in the present paper allow us to substantially simplify the construction in \cite{AL2} and, what is the main point, to avoid the self-intersections.

The proof of Theorem \ref{th:proper} will be a recursive application of the following approximation result.

\begin{lemma}
\label{lem:proper}
Let $A$ and $M$ be as in Theorem \ref{th:proper}. Let $U$ and $V$ be smoothly bounded compact domains such that $U\subset \mathring V\subset V\subset M$ and $U$ is Runge in $V$. Let $F=(F_1,F_2,\ldots,F_n)\in\mathscr I_A(U)$, let $\rho>0$, and assume that
\begin{equation}\label{eq:lem-proper}
\max\{|F_1(x)|,|F_2(x)|\}>\rho \quad \text{for all\ } x\in bU.
\end{equation}
Then there exists an $\wt F=(\wt F_1,\wt F_2,\ldots,\wt F_n)\in \mathfrak I_A(V)$ such that
\begin{itemize}
\item[\rm (i)] $\wt F$ is as close as desired to $F$ in the ${\mathscr C}^1(U)$ topology,
\item[\rm (ii)] $\max\{|\wt F_1(x)|,|\wt F_2(x)|\}>\rho$ for all $x \in V\setminus \mathring U$, and
\item[\rm (iii)] $\max\{|\wt F_1(x)|,|\wt F_2(x)|\}>\rho+1$ for all $x\in bV$.  
\end{itemize}
\end{lemma}


\begin{proof}
The conditions on $U$ and $V$ imply that these are sublevel sets of a strongly subharmonic Morse function $\tau$ defined on a neighborhood of $V$ in $M$. As in the proof of Theorem \ref{th:Mergelyan}, we obtain Lemma \ref{lem:proper} by a finite application of two special cases:  the {\em noncritical case} when there is no change of topology (i.e., $\tau$ has no critical points in $V\setminus U$), and the {\em critical case} when $\tau$ has a single critical point in $V\setminus U$.

\smallskip
\noindent {\it The noncritical case.} By Theorem \ref{th:Mergelyan} we may assume that $F$ extends to an $A$-immersion of an open neighborhood of $U$ in $M$. Denote by $\igoth$ the number of boundary components of $U$. Our conditions imply that
$
	V\setminus \mathring  U =\bigcup_{i=1}^\igoth \cA_i,
$
 where the $\cA_i$'s are pairwise disjoint compact annuli. For every $i\in\{1,\ldots,\igoth\}$ we denote by $\alpha_i$ the connected component of $b\cA_i$ contained in $bU$, and by $\beta_i$ the connected component of $b\cA_i$ contained in $bV$. Note that $\alpha_i$ and $\beta_i$ are smooth closed Jordan curves. 
 
 
It follows from \eqref{eq:lem-proper} that there exist $\jgoth\in\N$, subsets $I_1$ and $I_2$ of $I:=\{1,\ldots,\igoth\}\times \Z_\jgoth$ (here $\Z_\jgoth=\{0,\ldots,\jgoth-1\}$ denotes the additive cyclic group of integers modulus $\jgoth$), and a family of compact connected subarcs $\{\alpha_{i,j}\colon (i,j)\in I\}$, satisfying the following conditions:

\begin{itemize}
\item[\rm (a1)] $\bigcup_{j=1}^\jgoth \alpha_{i,j}=\alpha_i$.

\item[\rm (a2)] $\alpha_{i,j}$ and $\alpha_{i,j+1}$ have a common endpoint $p_{i,j}$ and are otherwise disjoint.

\item[\rm (a3)] $I_1\cup I_2=I$ and $I_1\cap I_2=\emptyset.$

\item[\rm (a4)] $|F_k(x)|>\rho$ for all $x\in\alpha_{i,j}$ and all $(i,j)\in I_k$, $k=1,2$.
\end{itemize}
From {\rm (a$4$)} one also has that
\begin{itemize}
\item[\rm (a$5$)] if $(i,j)\in I_h$ and $(i,j+1)\in I_l$, $h\neq l$, then $|F_k(p_{i,j})|>\rho$ for  $k\in\{1,2\}$.
\end{itemize}

For convenience we assume that $\jgoth\geq 3$. Next, for every $(i,j)\in I$ we choose a smooth embedded arc $\gamma_{i,j}\subset \cA_i$ with the following properties (see Fig.\ \ref{fig:S}):
\begin{itemize}
\item $\gamma_{i,j}$ is attached to $U$ at the endpoint $p_{i,j}$, it intersects the arc $\alpha_i$ transversely at that point, and $\gamma_{i,j}\cap \alpha_i=\{p_{i,j}\}$.

\item The other endpoint $q_{i,j}$ of the arc $\gamma_{i,j}$ lies in $\beta_i$, $\gamma_{i,j}$ intersects $\beta_i$ transversely  at that point, and $\gamma_{i,j}\cap \beta_i=\{q_{i,j}\}$.

\item The arcs $\gamma_{i,j}$, $j\in\Z_\jgoth$, are pairwise disjoint.
\end{itemize}

\begin{figure}[ht]
    \begin{center}
    \scalebox{0.25}{\includegraphics{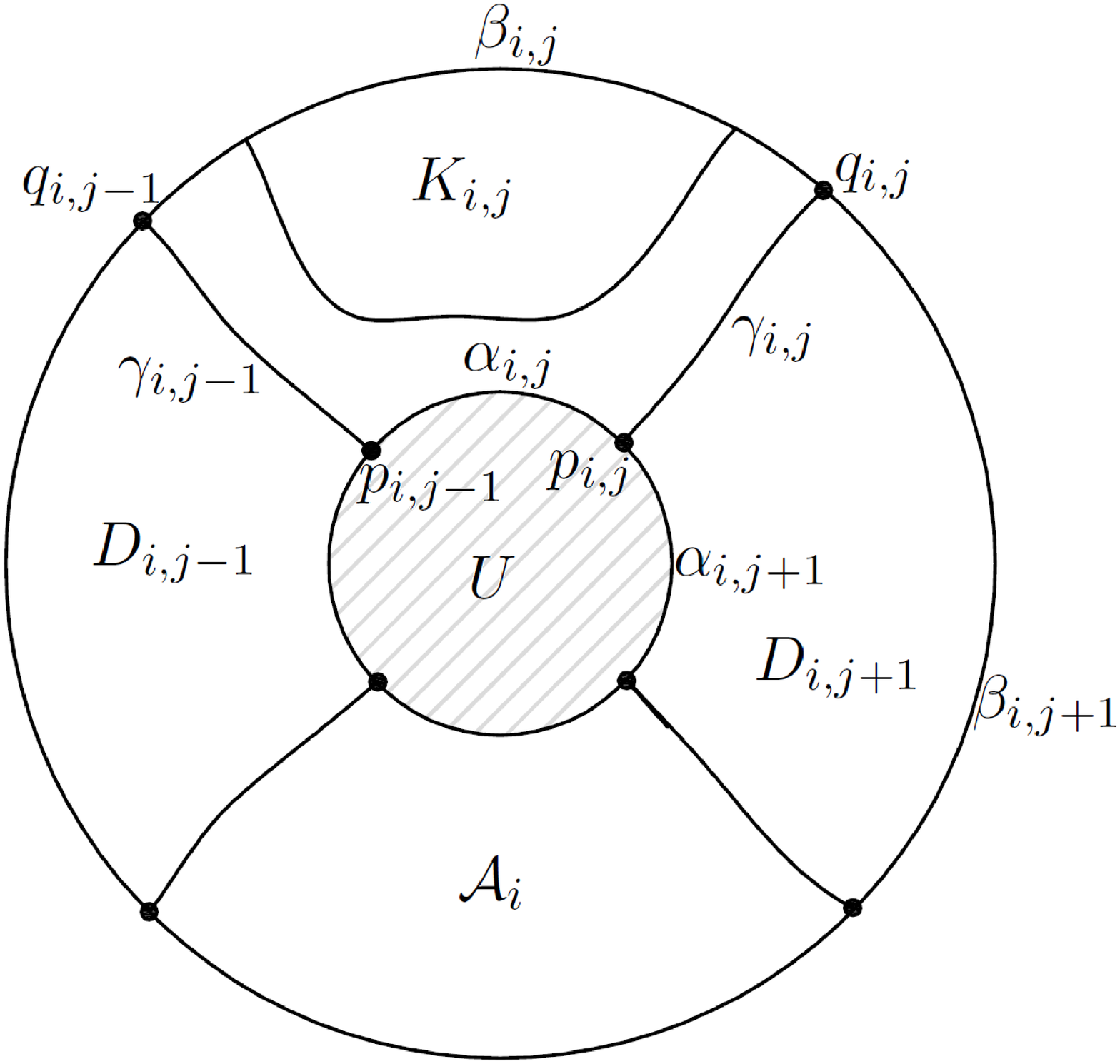}}
        \end{center}
        \vspace{-0.25cm}
\caption{The annulus $\cA_i$.}\label{fig:S}
\end{figure} 

Let $z=(z_1,\ldots,z_n)$ be the coordinates on $\C^n$. Recall that $\pi_k\colon \C^n\to \C$ is the $k$-th coordinate projection $\pi_k(z)=z_k$. Choose compact smooth embedded arcs $\lambda_{i,j}$ in $\C^n$, $(i,j)\in I$, meeting the following requirements:
\begin{itemize}
\item[\rm (b1)] $\lambda_{i,j}$ agrees with $F(\gamma_{i,j})$ near the endpoint $F(p_{i,j})$.

\item[\rm (b2)] $|\pi_k(z)|>\rho$ for all $z \in \lambda_{i,j-1}\cup \lambda_{i,j}$, $(i,j)\in I_k$, $k=1,2$.

\item[\rm (b3)] $|\pi_k(z)|>\rho+1$ for all $z \in\{v_{i,j-1},v_{i,j}\}$, $(i,j)\in I_k$, $k=1,2$, where $v_{i,l}\in\C^n$ denotes the other endpoint of the arc $\lambda_{i,l}$.

\item[\rm (b4)] The unit tangent vector field to $\lambda_{i,j}$ assumes values in $A\setminus\{0\}$.
\end{itemize}

To find such arcs $\lambda_{i,j}$, recall that the convex hull of $A\setminus\{0\}$ equals $\C^n$ (cf.\ Lemma \ref{lem:convexhull}) and use similar arguments as in the proof of Lemma \ref{lem:convex}. Since $p_{i,j-1}$ and $p_{i,j}$ are the endpoints of $\alpha_{i,j}$, properties {\rm (b1)} and {\rm (b2)} are compatible thanks to {\rm (a4)}. On the other hand, {\rm (b3)} is always possible.

As above, conditions {\rm (b$2$)} and {\rm (b3)} imply that,
\begin{itemize}
\item[\rm (b$5$)] if $(i,j)\in I_h$ and $(i,j+1)\in I_l$, $h\neq l$, then
\begin{itemize}
\item[$\bullet$] $|\pi_k(z)|>\rho$ for all $z\in\lambda_{i,j}$, and
\item[$\bullet$] $|\pi_k(v_{i,j})|>\rho+1$ for $k\in\{1,2\}$.
\end{itemize}
\end{itemize}

Taking into account {\rm (b1)}, we can find a smooth map $\widehat G\colon V \to\C^n$ that agrees with $F$ in an open neighborhood of $U$ and maps the arc $\gamma_{i,j}$ diffeomorphically onto $\lambda_{i,j}$ for every $(i,j)\in I$. The set 
\[
		S:=U\cup \big(\bigcup_{(i,j)\in I} \gamma_{i,j}\big) \subset M
\]
is admissible in the sense of Def.\ \ref{def:admissible}. Condition {\rm (b4)} shows that the map $\widehat G|_S$ is an $A$-immersion in the sense used in Theorem \ref{th:Mergelyan}. Hence Theorem \ref{th:Mergelyan}-(b) furnishes an $A$-immersion $G=(G_1,\ldots,G_n)\in\mathscr I_A(V)$ such that
\begin{itemize}
\item[\rm (c1)] $G$ is as close as desired to $F$ in the ${\mathscr C}^1(U)$ topology,

\item[\rm (c2)] $|G_k(x)|>\rho$ for all $x\in\gamma_{i,j-1}\cup\alpha_{i,j}\cup \gamma_{i,j}$, $(i,j)\in I_k$, $k=1,2$, and

\item[\rm (c3)] $|G_k(x)|>\rho+1$ for all $x\in\{q_{i,j-1},q_{i,j}\}$, $(i,j)\in I_k$, and $k=1,2$.
\end{itemize}

To obtain {\rm (c2)} we take into account {\rm (a4)} and {\rm (b2)}, whereas {\rm (c3)} follows from {\rm (b3)}. Furthermore, {\rm (c$2$)} and {\rm (c3)} give
\begin{itemize}
\item[\rm (c$4$)] if $(i,j)\in I_h$ and $(i,j+1)\in I_l$, $h\neq l$, then for $k\in\{1,2\}$ we have 
\begin{itemize}
\item[$\bullet$] $|G_k(x)|>\rho$ for all $x \in\gamma_{i,j}$, and
\item[$\bullet$] $|G_k(q_{i,j})|>\rho+1$.
\end{itemize}
\end{itemize}

Notice that $G$ satisfies condition {\rm (i)} in Lemma \ref{lem:proper}, whereas it satisfies {\rm (ii)} only on the arcs $\gamma_{i,j}$, and {\rm (iii)} only at the points $q_{i,j}$, $(i,j)\in I$. 
Therefore, $G$ meets all the requirements in Lemma \ref{lem:proper} on the admissible set $S$. 
%

Let $\beta_{i,j}$ be the compact connected Jordan arc in $\beta_i$ which connects the points 
$q_{i,j-1}$ and $q_{i,j}$ and does not intersect the set $\{q_{i,h}\colon h\in \Z_\jgoth\setminus \{j-1,j\}\}$; recall that $\jgoth\geq 3$. For every $(i,j)\in I$ we denote by $D_{i,j}$ the closed disc in $\cA_i$ bounded by the arcs $\alpha_{i,j}$, $\gamma_{i,j-1}$, $\gamma_{i,j}$, and $\beta_{i,j}$. (See Fig.\ \ref{fig:S}.) It is clear that $\cA_i=\bigcup_{j=1}^\jgoth D_{i,j}$ for every $i=1,\ldots,\igoth$. 

Since $G$ is continuous, properties {\rm (c2)} and {\rm (c3)} extend to small open neighborhoods of the compact sets $\gamma_{i,j-1}\cup \alpha_{i,j}\cup \gamma_{i,j}$ and $\{q_{i,j-1},q_{i,j}\}$, respectively. Therefore, we can choose for every $k\in\{1,2\}$ and every $(i,j)\in I_k$ a closed disc $K_{i,j}\subset D_{i,j}\setminus(\gamma_{i,j-1}\cup \alpha_{i,j}\cup \gamma_{i,j})$ such that
\begin{itemize}
\item[\rm (d1)] $K_{i,j}\cap \beta_{i,j}$ is a compact connected Jordan arc,
\item[\rm (d2)] $|G_k(x)|>\rho$ for all $x\in \overline{D_{i,j}\setminus K_{i,j}}$, and
\item[\rm (d3)] $|G_k(x)|>\rho+1$ for all $x\in \overline{\beta_{i,j}\setminus K_{i,j}}$.
\end{itemize}
(See Fig.\ \ref{fig:S}.) Obviously we have   
\begin{equation}\label{eq:Kij}
V\setminus \mathring U=\bigcup_{(i,j)\in I} \big( K_{i,j} \cup \overline{D_{i,j}\setminus K_{i,j}}\, \big)
\end{equation}
and
\begin{equation}\label{eq:betaij}
bV=\bigcup_{(i,j)\in I} \big((\beta_{i,j}\cap K_{i,j})\cup \overline{\beta_{i,j}\setminus K_{i,j}} \,\big).
\end{equation}
Assume without loss of generality that $I_1\neq \emptyset$; otherwise $I_2=I\neq\emptyset$ and we would reason in a symmetric way.

We now deform $G$ into another $A$-immersion $H\colon V\to\C^n$ satisfying Lemma \ref{lem:proper} on the set $U\cup \big(\bigcup_{(i,j)\in I_1} D_{i,j}\big)$. We will apply to $G$ a perturbation that is large in $K_{i,j}$ for all $(i,j)\in I_1$, small on $U$, and controlled elsewhere. Such control will be insured by demanding that $H_1=G_1$; see (\ref{eq:H1=G1}) below.

Observe that the compact set 
\[
	S_1:=
	U\cup \bigcup_{(i,j)\in I_2} D_{i,j}\,\,\, \cup \bigcup_{(i,j)\in I_1} K_{i,j}	\subset M
\]
is admissible (Def.\ \ref{def:admissible}). Note that the compact sets $U\cup \big(\bigcup_{(i,j)\in I_2} D_{i,j}\big)$ and $\bigcup_{(i,j)\in I_1} K_{i,j}$ are disjoint. Choose a point $\xi_1\in\C^n\cap\{z_1=0\}$ such that
\begin{equation}\label{eq:xi1}
|\pi_2(\xi_1)+G_2(x)|>\rho+1\quad \text{for all }x\in \bigcup_{(i,j)\in I_1}K_{i,j}.
\end{equation}
In fact, any $\xi_1$ with $|\pi_2(\xi_1)|$ large enough satisfies this condition. The map $\widehat H=(\widehat H_1,\ldots,\widehat H_n)\colon S_1\to\C^n$, given by
\begin{itemize}
\item[\rm (e1)] $\widehat H(x)=G(x)$ for all $x\in U\cup \big(\bigcup_{(i,j)\in I_2} D_{i,j}\big)$, and
\item[\rm (e2)] $\widehat H(x)=\xi_1+G(x)=(G_1(x),G_2(x)+\pi_2(\xi_1),\ldots,G_n(x)+\pi_n(\xi_1))$ for all $x\in \bigcup_{(i,j)\in I_1}K_{i,j}$,
\end{itemize}
is an $A$-immersion, and $\widehat H_1$ agrees with $G_1|_{S_1}$.
Therefore Theorem \ref{th:Mergelyan2} applies and provides an $A$-immersion $H=(H_1,\ldots,H_n)\in \mathscr I_A(V)$ such that 
\begin{equation}\label{eq:H1=G1}
H_1=G_1
\end{equation}
and the following properties hold:
\begin{itemize}
\item[\rm (f1)] $H$ is as close to $\wh H$ as desired in the ${\mathscr C}^1(S_1)$ topology.
\item[\rm (f2)] $|H_k(x)|>\rho$ for all $x\in\overline{D_{i,j}\setminus K_{i,j}}$, $(i,j)\in I_k$, $k=1,2$.
\item[\rm (f3)] $|H_k(x)|>\rho+1$ for all $x\in \overline{\beta_{i,j}\setminus K_{i,j}}$, $(i,j)\in I_k$, $k=1,2$.
\item[\rm (f4)] $|H_2(x)|>\rho+1$ for all $x\in K_{i,j}$ and all $(i,j)\in I_1=I\setminus I_2$.
\end{itemize}

Indeed, for $k=1$, properties {\rm (f2)} and {\rm (f3)} follow from {\rm (d2)}, {\rm (d3)}, and \eqref{eq:H1=G1}; whereas for $k=2$ they are guaranteed by {\rm (d2)}, {\rm (d3)}, and {\rm (e1)}, provided that $H$ is sufficiently close to $\widehat H$ on $S_1$. Likewise, {\rm (f4)} is implied by {\rm (e2)} and \eqref{eq:xi1} provided that the approximation is sufficiently close.

If $I_2=\emptyset$, then the proof of Lemma \ref{lem:proper} is already done. Indeed, in such case $I=I_1$; hence, taking into account \eqref{eq:Kij}, properties {\rm (f2)} and {\rm (f4)} imply item {\rm (ii)} in the lemma. Likewise,
{\rm (iii)} follows from {\rm (f3)}, {\rm (f4)}, and \eqref{eq:betaij}. Finally, item {\rm (i)} is insured by {\rm (f1)} and {\rm (c1)}.

Assume now that $I_2\neq \emptyset$. In the next step we deform $H$ to obtain the $A$-immersion $\wt F\in \mathscr I_A(V)$ satisfying Lemma \ref{lem:proper}. The deformation is big on $\bigcup_{(i,j)\in I_2} D_{i,j}$ and small elsewhere. The deformation procedure will be symmetric to the one in the previous step of the proof. Consider the set
\[
S_2:= U\cup \bigcup_{(i,j)\in I_1} D_{i,j}  \,\,\,\cup \bigcup_{(i,j)\in I_2} K_{i,j},
\]
and choose a point $\xi_2\in\C^n\cap\{z_2=0\}$ such that
\begin{equation}\label{eq:xi2}
|\pi_1(\xi_2)+H_1(x)|>\rho+1\quad \text{for all }x\in \bigcup_{(i,j)\in I_2}K_{i,j}.
\end{equation}
Consider the $A$-immersion $Y=(Y_1,\ldots,Y_n)\colon S_2\to\C^n$ given by
\begin{itemize}
\item[\rm (g1)] $Y(x)=H(x)$ for all $x\in U\cup \big(\bigcup_{(i,j)\in I_1} D_{i,j}\big)$, and
\item[\rm (g2)] $Y(x)=\xi_2+H(x)=(H_1(x)+\pi_1(\xi_2),H_2(x),\ldots,H_n(x)+\pi_n(\xi_2))$ for all $x\in \bigcup_{(i,j)\in I_2}K_{i,j}$.
\end{itemize}
Apply Theorem \ref{th:Mergelyan2} to get $\wt F=(\wt F_1,\ldots,\wt F_n)\in \mathscr I_A(V)$, with
\begin{equation}\label{eq:F2=H2}
\wt F_2=H_2
\end{equation}
and such that the following conditions hold:
\begin{itemize}
\item[\rm (h1)] $\wt F$ is as close to $Y$ as desired in the ${\mathscr C}^1(S_2)$ topology.
\item[\rm (h2)] $|\wt F_k(x)|>\rho$ for all $x\in\overline{D_{i,j}\setminus K_{i,j}}$, for all $(i,j)\in I_k$, $k=1,2$.
\item[\rm (h3)] $|\wt F_k(x)|>\rho+1$ for all $x\in \overline{\beta_{i,j}\setminus K_{i,j}}$, for all $(i,j)\in I_k$, $k=1,2$.
\item[\rm (h4)] $|\wt F_k(x)|>\rho+1$ for all $x\in K_{i,j}$, for all $(i,j)\in I\setminus I_k$, $k=1,2$.
\end{itemize}

Indeed, for $k=2$, properties {\rm (h2)} and {\rm (h3)} follow from {\rm (f2)}, {\rm (f3)}, and \eqref{eq:F2=H2}; whereas for $k=1$ they are insured by {\rm (f2)}, {\rm (f3)}, and {\rm (h1)}, provided that the approximation of $Y$ by $\wt F$ is close enough on the set $S_2$. Likewise, for $k=2$, {\rm (h4)} is implied by {\rm (f4)} and \eqref{eq:F2=H2}; whereas for $k=1$ it follows from {\rm (g2)} and \eqref{eq:xi2}, if the approximation is sufficiently close.

To see that $\wt F$ satisfies Lemma \ref{lem:proper}, notice that {\rm (i)} is insured by {\rm (h1)}, {\rm (f1)}, and {\rm (c1)}; condition {\rm (ii)} follows from \eqref{eq:Kij}, {\rm (h2)}, and {\rm (h4)}; whereas \eqref{eq:betaij}, {\rm (h3)}, and {\rm (h4)} guarantee the condition {\rm (iii)}. 

This concludes the proof in the noncritical case.

\smallskip
\noindent {\it The critical  case.} We assume that $\tau$ has a unique (Morse) critical point $p$ in $V\setminus U$. This point has Morse index 0 or 1. 

If the Morse index of $p$ equals 0, a new connected component of the sublevel set $\{\tau\le t\}$ appears at $p$ when $t$ passes the value $\tau(p)$, and it is trivial to find a map $\wt F$ satisfying the Lemma on this new component.

If the Morse index of $p$ equals 1, there is a compact Jordan arc $\gamma \subset \mathring V\setminus \mathring U$,  attached with both endpoints to $U$, such that $S:=U\cup\gamma$ is an admissible Runge set in $V$  (Def.\ \ref{def:admissible}) and a strong deformation retract of $V$. Clearly $F$ extends to an $A$-immersion $F\colon S\to\C^n$ satisfying $\max\{|F_1(x)|,|F_2(x)|\} > \rho$ for all $x\in\gamma$; see \eqref{eq:lem-proper}. Theorem \ref{th:Mergelyan}-{\rm (a)} furnishes a smoothly bounded compact domain, $W$, and an $A$-immersion $G=(G_1,\ldots,G_n)\in\mathscr I_A(W)$, such that $U\subset\mathring W\subset W\subset \mathring V$, 
$V$ is a noncritical extension of $W$, $G$ is as close as desired to $F$ in the ${\mathscr C}^1(S)$ topology, and 
$\max\{|G_1(x)|,|G_2(x)|\}>\rho$ for $x\in bW$. This reduces the proof to the noncritical case. 
\end{proof}

\smallskip
\noindent {\it Proof of Theorem \ref{th:proper}.} 
Let $F=(F_1,F_2,\ldots,F_n)\colon M_0\to\C^n$ be an $A$-immersion, where $M_0$ is a smoothly bounded compact Runge domain in $M$ containing $K$ in its interior. By general position we may assume that the set $(F_1,F_2)(bM_0)\subset\C^2$ does not contain the origin of $\C^2$. Choose a number $\xi>0$ such that 
\begin{equation}\label{eq:basis}
\max\{|F_1(x)|,|F_2(x)|\}>\xi\quad \text{for all $x\in bM_0$}.
\end{equation}
Pick a number $\epsilon>0$ and set $F^0:=F$ and $\epsilon_0:=\epsilon$. Choose an exhaustion of $M$ by a sequence $M_0\subset M_1\subset M_2 \subset \cdots \bigcup_{j=0}^\infty M_j= M$ of smoothly bounded compact Runge domains. 
A recursive application of Lemma \ref{lem:proper} gives numbers $\epsilon_j>0$ and $A$-immer\-sions $F^j=(F^j_1,F^j_2,\ldots,F^j_n)\colon M_j\to\C^n$ such that the following conditions hold for every $j\in\N$:
\begin{itemize}
\item[\rm (a)] $\|F^j-F^{j-1}\|_{M_{j-1}}:=\max_{x\in M_{j-1}}|F^j(x)-F^{j-1}(x)|<\epsilon_{j-1}/2$,
\item[\rm (b)] $\max\{|F^j_1(x)|,|F^j_2(x)|\}>j-1$ for all $x \in M_j\setminus \mathring M_{j-1}$, 
\item[\rm (c)] $\max\{|F^j_1(x)|,|F^j_2(x)|\}>j$ for all $x\in bM_j$, and
\item[\rm (d)] $0<\epsilon_j<\epsilon_{j-1}/4$,
\end{itemize}
The induction begins thanks to \eqref{eq:basis}, and the inductive step is guaranteed by property {\rm (c)}. Furthermore, by Theorem \ref{th:desing} we may assume that each $F^j$ is actually an $A$-embedding and
\begin{itemize}
\item[\rm (e)] every holomorphic map $G\colon M\to\C^n$ satisfying $\|G-F^j\|_{M_j}<\epsilon_j$ is an embedding on $M_{j-1}$.
\end{itemize}
As in the proof of Theorem \ref{th:desing2}, properties {\rm (a)}, {\rm (d)}, and {\rm (e)} ensure that the limit map $\wt F=(\wt F_1,\wt F_2,\ldots,\wt F_n)=\lim_{j\to\infty} F^j\colon M\to \C^n$ exists, is an $A$-embedding, and satisfies 
$\|\wt F-F^j\|_{M_j}<\epsilon$ for all $j=0,1,...$. 
In particular we have $\|\wt F-F\|_{M_0}<\epsilon$. Together with property {\rm (b)} we get
\[
	\max\{|\wt F_1(x)|,|\wt F_2(x)|\}>j-1-\epsilon \quad 
	\text{for all } x\in M_j\setminus \mathring M_{j-1},\ j\in \N.
\]
This shows that the map $(\wt F_1,\wt F_2)\colon M\to\C^2$ is proper.
\qed


\medskip
{\bf Acknowledgements.} 
\rm A.\ Alarc\'{o}n is supported by Vicerrectorado de Pol\'{i}tica Cient\'{i}fica e Investigaci\'{o}n de la Universidad de Granada, and is partially supported by MCYT-FEDER grants MTM2007-61775 and MTM2011-22547, Junta de Andaluc\'{i}a Grant P09-FQM-5088, and the grant PYR-2012-3 CEI BioTIC GENIL (CEB09-0010) of the MICINN CEI Program. 

F.\ Forstneri\v c is supported by the research program P1-0291 from ARRS, Republic of Slovenia.

We wish to thank the anonymous referee for the useful remarks which helped us to improve the presentation.


\bibliographystyle{amsplain}

\begin{thebibliography}{10}

\bibitem[Abh]{Abraham}
{\scshape Abraham, R.}:
Transversality in manifolds of mappings.
Bull.\ Amer.\ Math.\ Soc., \textbf{69}, 470--474 (1963)

\bibitem[AS]{AS}
{\scshape Ahlfors, L.V.; Sario, L.}: 
Riemann surfaces.
Princeton Mathematical Series, No.\ 26, Princeton University Press, Princeton (1960) 

\bibitem[Ala]{Al}
{\scshape Alarc\'on, A.}: 
Compact complete minimal immersions in $\R^3$.
Trans. Amer. Math. Soc. \textbf{362}, 4063--7076 (2010)

\bibitem[AFM]{AFM}
{\scshape Alarc\'on, A.; Ferrer, L.; Mart\'in, F.}: 
Density theorems for complete minimal surfaces in $\R^3$.
Geom. Funct. Anal. \textbf{18}, 1--49 (2008) 

\bibitem[AF]{AF}
{\scshape Alarc\'on, A.; Forstneri\v c, F.}: 
Every bordered Riemann surface is a complete proper curve in a ball. 
Math.\ Ann. (2013).  

\texttt{http://link.springer.com/article/10.1007/s00208-013-0931-4}

\bibitem[AL1]{AL1}
{\scshape Alarc\'on, A.; L\'opez F.J.}: 
Null Curves in $\mathbb{C}^3$ and Calabi-Yau Conjectures.
Math. Ann. \textbf{355}, 429--455 (2013) 

\bibitem[AL2]{AL2}
{\scshape Alarc\'on, A.; L\'opez, F.J.}: 
Minimal surfaces in $\mathbb{R}^3$ properly projecting into $\mathbb{R}^2$.
J. Diff.\ Geom. \textbf{90}, 351--382 (2012) 

\bibitem[AL3]{AL3}
{\scshape Alarc\'on, A.; L\'opez, F.J.}: 
Compact complete null curves in complex 3-space. 
Israel J. Math., in press. \texttt{arXiv:1106.0684}

\bibitem[Bry]{Br}
{\scshape Bryant, R.}: 
Surfaces of mean curvature one in hyperbolic space. 
Th\'eorie des vari\'et\'es minimales et applications (Palaiseau, 1983--1984).
Ast\'erisque \textbf{154-155} (1987), 12, 321--347, 353 (1988)

\bibitem[Chi]{Chirka}
{\scshape Chirka, E.M.}: 
Complex analytic sets.
Kluwer, Dordrecht (1989)

\bibitem[CM]{CM} 
{\scshape Colding, T.H.; Minicozzi II, W.P.}: 
The Calabi-Yau conjectures for embedded surfaces. 
Ann. Math., (2) \textbf{167}, 211--243 (2008)

\bibitem[CHR]{CHR}
{\scshape Collin, P.; Hauswirth, L.; Rosenberg, H.}: 
The geometry of finite topology Bryant surfaces. 
Ann. Math., (2) \textbf{153}, 623--659 (2001)

\bibitem[DF]{DF}
{\scshape Drinovec Drnov\v sek, B.; Forstneri\v c, F.}: Holomorphic curves in complex spaces.
Duke Math. J. \textbf{139}, 203--254 (2007)

\bibitem[EM]{EM}
{\scshape Eliashberg, Y.; Mishachev, N.}:
Introduction to the $h$-principle.
Graduate Studies in Math., 48.
Amer.\ Math.\ Soc., Providence (2002)

\bibitem[FMM]{FMM}
{\scshape Ferrer, L.; Mart\'in, F.; Meeks III, W.H.}: 
Existence of proper minimal surfaces of arbitrary topological type.
Adv. Math. \textbf{231}, 378--413 (2012)

\bibitem[Fr]{Forster-book}
{\scshape Forster, O.}:
Lectures on Riemann surfaces. 
Graduate Texts in Mathematics, 81.
Springer-Verlag, New York (1991) 

\bibitem[F1]{FF:subelliptic}
{\scshape Forstneri\v c, F.}:
The Oka principle for sections of subelliptic submersions.
Math.\ Z.\ \textbf{241}, 527--551 (2002)

\bibitem[F2]{FF:multivalued}
{\scshape Forstneri\v c, F.}:
The Oka principle for multivalued sections of ramified mappings.
Forum Math. \textbf{15}, 309--328 (2003)

\bibitem[F3]{FF:manifolds}
{\scshape Forstneri\v c, F.}:
Manifolds of holomorphic mappings from strongly pseudoconvex domains. 
Asian J.\ Math. \textbf{11}, 113-–126 (2007)

\bibitem[F4]{FF:Oka} 
{\scshape Forstneri\v c, F.}:
Oka manifolds.
C.\ R.\ Acad.\ Sci.\ Paris, Ser.\ I \textbf{347}, 1017-–1020 (2009)

\bibitem[F5]{F:book}
{\scshape Forstneri\v c, F.}:
Stein Manifolds and Holomorphic Mappings
(The Homotopy Principle in Complex Analysis). 
Ergebnisse der Mathematik und ihrer Grenzgebiete, 
3.\ Folge, 56. 
Springer-Verlag, Berlin--Heidelberg (2011)

\bibitem[FL1]{FL1}
{\scshape Forstneri\v{c}, F.; L\'arusson, F.}:
Survey of Oka theory.
New York J.\ Math. \textbf{17a}, 1--28.  (2011)

\bibitem[FL2]{FL2}
{\scshape Forstneri\v{c}, F.:} Oka manifolds: From Oka to Stein and back.
With an appendix by F.\ L\'arusson. 
\texttt{arxiv:1211.6383}

\bibitem[FL3]{FL3}
{\scshape Forstneri\v{c}, F.; L\'arusson, F.}:
Holomorphic flexibility properties of compact complex surfaces.
Int. Math. Res. Notices IMRN (2013). 

\texttt{http://dx.doi.org/10.1093/imrn/rnt044}

\bibitem[Gra2]{Gra1}
{\scshape Grauert, H.}:
Approximationss\"atze f\"ur holomorphe Funktionen mit Werten in komplexen Lieschen Gruppen.
Math.\ Ann. \textbf{133}, 450--472 (1957)

\bibitem[Gra1]{Gra2}
{\scshape Grauert, H.}:
Approximationss\"atze f\"ur holomorphe Funktionen mit Werten in komplexen R\"aumen.
Math.\ Ann. \textbf{133}, 139--159 (1957)

\bibitem[Gra3]{Gra3}
{\scshape Grauert, H.}:
Analytische Faserungen \"uber holomorph-vollst\"andigen R\"aumen.
Math.\ Ann. \textbf{135}, 263--273 (1958)

\bibitem[G1]{Gromov:convex}
{\scshape  Gromov, M.}: 
Convex integration of differential relations, I.
Izv.\ Akad.\ Nauk SSSR Ser.\ Mat. \textbf{37}, 329--343 (1973) (Russian).
English transl.: Math.\ USSR Izv. \textbf{37} (1973)

\bibitem[G2]{Gromov:PDR}
{\scshape Gromov, M.}:  
Partial differential relations.
Ergebnisse der Mathematik und ihrer Grenzgebiete, 3.\ Folge, 9.
Springer-Verlag, Berlin--New York (1986)

\bibitem[G3]{Gromov:Oka}
{\scshape Gromov, M.}:
Oka's principle for holomorphic sections of elliptic bundles.
J.\ Amer.\ Math.\ Soc.\ \textbf{2}, 851--897 (1989)

\bibitem[GN]{GN}
{\scshape Gunning, R.C.; Narasimhan, R.}:
Immersion of open Riemann surfaces.
Math. Ann. \textbf{174}, 103--108 (1967)

\bibitem[Hor]{Hormander-SCV}
{\scshape H\"ormander, L.}:
An Introduction to Complex Analysis in Several Variables. Third edn.
North-Holland Mathematical Library, 7. 
North Holland, Amsterdam (1990)

\bibitem[Lem]{Lempert}
{\scshape Lempert, L.}: 
The Dolbeault complex in infinite dimensions. I. 
J.\ Amer.\ Math.\ Soc.\ \textbf{11}, 485–-520 (1998) 

\bibitem[LMM]{LMM}
{\scshape  L\'opez, F.J.; Mart\'in, F.; Morales, S.}:
Adding handles to Nadirashvili's surfaces. 
J.\ Diff.\ Geom. \textbf{60}, 155--175 (2002)

\bibitem[LR]{LR}
{\scshape L\'opez, F.J.; Ros, A.}: 
On embedded complete minimal surfaces of genus zero. 
J. Diff. Geom. \textbf{33}, 293--300 (1991) 

\bibitem[Maj]{Majcen-2007}
{\scshape Majcen, I.}:
Closed holomorphic 1-forms without zeros on Stein manifolds.  
Math.\ Z. \textbf{257}, 925--937  (2007)

\bibitem[MUY]{MUY1} 
{\scshape Mart\'{i}n, F.; Umehara, M.; Yamada, K.}: 
Complete bounded null curves immersed in $\C^3$ and $SL(2,\C)$. 
Calc. Var. Partial Differential Equations \textbf{36}, 119--139 (2009). 
Erratum: Complete bounded null curves immersed in $\C^3$ and $SL(2,\C)$. 
Calc.\ Var.\ Partial Differential Equations \textbf{46}, 439--440 (2013)

\bibitem[MP1]{MP1}
{\scshape Meeks III, W.H.; P\'erez, J.}: 
The classical theory of minimal surfaces.
Bull. Amer. Math. Soc. (N.S.) \textbf{48}, 325--407 (2011)

\bibitem[MP2]{MP2}
{\scshape Meeks III, W.H.; P\'erez, J.}: 
A survey on classical minimal surface theory.
University Lecture Series, Amer.\ Math.\ Soc., in press. 


\bibitem[MPR]{MPR} 
{\scshape Meeks III, W.H.; P\'{e}rez, J.; Ros, A.}: 
The embedded Calabi-Yau conjectures for finite genus. 
\texttt{http://www.ugr.es/~jperez/papers/papers.htm}

\bibitem[Nad]{Na} 
{\scshape Nadirashvili, N.}: 
Hadamard's and Calabi-Yau's conjectures on negatively curved and minimal surfaces.
Invent. Math. \textbf{126}, 457--465 (1996)

\bibitem[Oka]{Oka}
{\scshape Oka, K.}:
Sur les fonctions des plusieurs variables.
III: Deuxi\`eme probl\`eme de Cousin.
J.\ Sci.\ Hiroshima Univ. \textbf{9}, 7--19 (1939)

\bibitem[Oss]{Osserman}
{\scshape Osserman, R.}: 
A survey of minimal surfaces. Second edn. 
Dover Publications, Inc., New York (1986)

\bibitem[Ros]{Ro}
{\scshape Rosenberg, H.}: 
Bryant surfaces. 
In: The global theory of minimal surfaces in flat spaces. Lectures given at the 2nd C.I.M.E. Session held in Martina Franca, July 7–-14, 1999. Lecture Notes in Math., vol. 1775, pp. 67-–111.
Springer-Verlag, Berlin (2002) 

\bibitem[UY]{UY}
{\scshape Umehara, M.; Yamada, K.}: 
Complete surfaces of constant mean curvature $1$ in the hyperbolic $3$-space.
Ann. Math., (2) \textbf{137}, 611--638 (1993) 


\end{thebibliography}

\end{document}